\documentclass[reqno]{amsart}
\usepackage[english]{babel}
\usepackage{amssymb,amsmath,hyperref}
\usepackage{amsrefs}
\usepackage[foot]{amsaddr}
\usepackage{mathrsfs}
\usepackage{cleveref, enumerate}

\usepackage{transfer}

\newtheorem{thm}{Theorem}

\newtheorem{cor}[thm]{Corollary}
\newtheorem{defi}[thm]{Definition}
\newtheorem{rem}[thm]{Remark}
\newtheorem{nota}[thm]{Notation}

\newtheorem{princ}[thm]{Principle}

\newtheorem{ack}[thm]{Acknowledgement}

\newtheorem*{tempo*}{Template}

\newcommand\be{\begin{equation}}
\newcommand\ee{\end{equation}}

\newbox\gnBoxA
\newdimen\gnCornerHgt
\setbox\gnBoxA=\hbox{$\ulcorner$}
\global\gnCornerHgt=\ht\gnBoxA
\newdimen\gnArgHgt

\def\Godelnum #1{%
	\setbox\gnBoxA=\hbox{$#1$}%
	\gnArgHgt=\ht\gnBoxA%
	\ifnum \gnArgHgt<\gnCornerHgt
		\gnArgHgt=0pt%
	\else
		\advance \gnArgHgt by -\gnCornerHgt%
	\fi
	\raise\gnArgHgt\hbox{$\ulcorner$} \box\gnBoxA %
		\raise\gnArgHgt\hbox{$\urcorner$}}
\def\bdefi{\begin{defi}\rm}
\def\edefi{\end{defi}}
\def\bnota{\begin{nota}\rm}
\def\enota{\end{nota}}
\def\brem{\begin{rem}\rm}
\def\erem{\end{rem}}

\def\DNR{\textup{DNR}}

\def\RCA{\textup{\textsf{RCA}}}

\def\RCAo{\textup{\textsf{RCA}}_{0}^{\omega}}
\def\RCAO{\textup{\textsf{RCA}}_{0}^{\Lambda}}

\def\WKL{\textup{\textsf{WKL}}}

\def\UWKL{\textup{\textsf{UWKL}}}

\def\T{\mathcal{T}}

\def\bye{\end{document}}

\def\P{\textup{\textsf{P}}}

\def\N{{\mathbb  N}}

\def\({\textup{(}}
\def\){\textup{)}}

\def\st{\textup{st}}
\def\asa{\leftrightarrow}

\def\di{\rightarrow}

\def\eps{\varepsilon}
\def\epsa{\varepsilon_{\alpha}}

\def\ACA{\textup{\textsf{ACA}}}
\def\paai{\Pi_{1}^{0}\textup{-\textsf{TRANS}}}

\def\QFAC{\textup{\textsf{QF-AC}}}

\def\ADS{\textup{ADS}}
\def\RT{\textup{\textsf{RT}}}
\def\UADS{\textup{\textsf{UADS}}}

\def\RS{{\mathfrak{RS}}}

\def\MU{\textup{\textsf{MU}}}

\def\HAC{\textup{\textsf{HAC}}}

\def\INT{\textup{\textsf{int}}}

\def\field{\textup{field}}

\def\DNR{\textup{\textsf{DNR}}}
\def\ADS{\textup{\textsf{ADS}}}
\def\CAC{\textup{\textsf{CAC}}}
\def\UCAC{\textup{\textsf{UCAC}}}
\def\UDNR{\textup{\textsf{UDNR}}}
\def\UCOH{\textup{\textsf{UCOH}}}
\def\COH{\textup{\textsf{COH}}}

\numberwithin{equation}{section}
\numberwithin{thm}{section}

\usepackage{comment}

\begin{document}
\title[Taming the Reverse Mathematics zoo]{The Taming of the Reverse Mathematics zoo}

\begin{abstract}
Reverse Mathematics is a program in the foundations of mathematics.  Its results give rise to an elegant classification of theorems of ordinary mathematics based on computability.  
In particular, the majority of these theorems fall into \emph{only five} categories of which the associated logical systems are dubbed `the Big Five'.  
Recently, a lot of effort has been directed towards finding \emph{exceptional} theorems, i.e.\ which fall \emph{outside} the Big Five categories.  
The so-called Reverse Mathematics zoo is a collection of such exceptional theorems (and their relations).    
In this paper, we show that the \emph{uniform} versions of the zoo-theorems, i.e.\ where a functional computes the objects stated to exist, all fall in the third Big Five category \emph{arithmetical comprehension}, inside Kohlenbach's {higher-order} Reverse Mathematics.  
In other words, the zoo seems to disappear at the uniform level.      
Our classification applies to all theorems whose objects exhibit \emph{little structure}, a notion we conjecture to be connected to Montalb\'an's notion \emph{robustness}.  
Surprisingly, our methodology reveals a hitherto unknown `computational' aspect of Nonstandard Analysis:  We shall formulate an \emph{algorithm} $\RS$ which takes as input the proof of a specific equivalence \emph{in Nelson's internal set theory}, and outputs the proof of the desired equivalence (not involving Nonstandard Analysis) between the uniform zoo principle and arithmetical comprehension.  
Moreover, the equivalences thus proved are even \emph{explicit}, i.e.\ a term from the language converts the functional from one uniform principle into the functional from the other one and vice versa.    
\end{abstract}
\author{Sam Sanders}
\address{Department of Mathematics, Ghent University, Belgium \& Munich Center for Mathematical Philosophy, LMU Munich, Germany}
\email{sasander@me.com}
\maketitle
\thispagestyle{empty}


\section{Introduction: Reverse Mathematics and its zoo}\label{intro}
In two words, the subject of this paper is the \emph{Reverse Mathematics} classification in Kohlenbach's framework (\cite{kohlenbach2}) of uniform versions of principles from the \emph{Reverse Mathematics zoo} (\cite{damirzoo}), 
namely as equivalent to \emph{arithmetical comprehension}.   We first discuss the italicised notions in more detail.

\medskip

For an introduction to the foundational program Reverse Mathematics (RM for short), we refer to \cites{simpson2, simpson1}.  
One of the main results of RM is that the majority of theorems from \emph{ordinary mathematics}, i.e.\ about countable and separable objects, fall into \emph{only five} 
categories of which the associated logical systems are dubbed `the Big Five' (See e.g.\ \cite{montahue}*{p.\ 432}).  
In the last decade or so, a huge amount of time and effort was invested in identifying theorems falling \emph{outside} of the Big Five categories.  
All such exceptional theorems (and their relations) falling below the third Big Five system, are collected in the so-called RM zoo (See \cite{damirzoo}).   

\medskip

In this paper, we shall establish that the exceptional principles inhabiting the RM zoo become \emph{non-exceptional at the uniform level}, namely that the uniform versions of RM zoo-principles are all equivalent to arithmetical comprehension, the aforementioned third Big Five system of RM. 
As a first example of such a `uniform version', consider the principle \textsf{UDNR}, to be studied in Section~\ref{CUDNR}.  
\be\label{UDNR2}\tag{\textsf{UDNR}}
(\exists \Psi^{1\di1})\big[(\forall A^{1})(\forall e^{0})(\Psi(A)(e)\ne \Phi_{e}^{A}(e))\big].
\ee
Clearly, $\UDNR$ is the uniform version of the zoo principle\footnote{We sometimes refer to inhabitants of the RM zoo as `theorems' and sometimes as `principles'.} $\DNR$, defined as:  
\be\label{DNR2}\tag{\textup{\DNR}}
(\forall A^{1})(\exists f^{1})(\forall e^{0})\big[f(e)\ne \Phi_{e}^{A}(e)\big].
\ee  
The principle $\DNR$ was first formulated in \cite{withgusto} and is even strictly implied by $\textsf{WWKL}$ (See \cite{compdnr}) where the latter principle sports \emph{some} Reverse Mathematics equivalences (\cites{montahue, yuppie, yussie}) but is not a Big Five system.  Nonetheless, we shall prove that $\UDNR\asa (\exists^{2})$, where the second principle is the functional version of arithmetical comprehension, the third Big Five system of RM.    
In other words, \emph{the `exceptional' status of $\DNR$ disappears completely if we consider its uniform version $\UDNR$}.  
The proof of the equivalence $\UDNR\asa (\exists^{2})$ takes place in $\RCAo$ (See Section \ref{base}), the base theory of Kohlenbach's \emph{higher-order Reverse Mathematics}.  

\medskip

More generally, in Sections \ref{CUDNR}, \ref{czoo}, and \ref{DX}, we show that a number of uniform zoo-principles are equivalent to arithmetical comprehension inside $\RCAo$.  
In Section \ref{GT}, we formulate a general\footnote{For instance, as shown in Section~\ref{DX}, our template is certainly not limited to $\Pi_{2}^{1}$-formulas, and surprisingly even applies to \emph{contrapositions} of RM zoo principles, including the Ramsey theorems.} template for classifying (past and future) zoo-principles in the same way.    
As will become clear, our template provides a \emph{uniform and elegant} approach to classifying uniform principles originating from the RM zoo;  In other words, the RM \emph{zoo seems to disappear at the uniform level} (but see Remark~\ref{dazoo}).  
As to a possible explanation for this phenomenon, the axiom of \emph{extensionality} plays a central role in our template, as discussed in Remark~\ref{equ3}.  
Another key ingredient of the template is the presence of `little structure' (which is e.g.\ typical of statements from combinatorics) on the objects in RM zoo principles, which gives rise to \emph{non-robust} theorems in the sense of Montalb\'an (\cite{montahue}), as discussed in Section \ref{robu}.   

\medskip

To obtain the aforementioned equivalences, Nonstandard Analysis in the form of Nelson's \emph{internal set theory} (\cite{wownelly}), is used \emph{as a tool} in this paper.   
In particular, these equivalences are formulated as theorems of Kohlenbach's base theory $\RCAo$ (See \cite{kohlenbach2} and Section \ref{popo}), 
and are obtained by applying the algorithm $\RS$ (See Section~\ref{GT}) to associated equivalences \emph{in Nonstandard Analysis}.  
Besides providing a streamlined and uniform approach, the use of Nonstandard Analysis via $\RS$ also results in \emph{explicit\footnote{An implication $(\exists \Phi)A(\Phi)\di (\exists \Psi)B(\Psi)$ is \emph{explicit} if there is a term $t$ in the language such that additionally $(\forall \Phi)[A(\Phi)\di B(t(\Phi))]$, i.e.\ $\Psi$ can be explicitly defined in terms of $\Phi$.\label{dirkske}}} equivalences \emph{without extra effort}.  In particular, we shall just prove equivalences inside Nonstandard Analysis \emph{without paying any attention to effective content}, 
and extract the explicit equivalences using the algorithm $\RS$.   This hitherto unknown `computational aspect' of Nonstandard Analysis is perhaps the true surprise of this paper.  

\medskip

Finally, as to conceptual considerations, the above-mentioned `disappearance' of the RM zoo suggests that Kohlenbach's \emph{higher-order} RM (\cite{kohlenbach2}) is not just `RM with higher types', 
but a separate field of study giving rise to a completely different classification;  In particular, the latter comes equipped with its own notion of exceptionality, notably different from the one present in Friedman-Simpson-style RM.  
In light of the results in Section~\ref{DX}, one could go even as far as saying that, at the uniform level, \emph{weak K\"onig's lemma is more exceptional than e.g.\ Ramsey's theorem for pairs}, 
as the latter is more robust (at the uniform level) than the former, due to the behaviour of their contrapositions (at the uniform level).  As the saying (sort of) goes, one man's exception is another's mainstream.  

\medskip

In conclusion, the stark contrast in exceptional behaviour between principles from the RM zoo and their uniform counterparts, speaks in favour of the study of higher-order RM.    
Notwithstanding the foregoing, `unconditional' arguments for the study of higher-order RM are also available, as discussed in Section \ref{whereohwhere}.

\section{About and around internal set theory}\label{base}
In this section, we introduce Nelson's \emph{internal set theory}, first introduced in \cite{wownelly}, and its fragment $\P$ from \cite{brie}.
We shall also introduce Kohlenbach's base theory $\RCAo$ from \cite{kohlenbach2}, and the system $\RCAO$, which is based on $\P$.   
\subsection{Introduction: Internal set theory}
In Nelson's \emph{syntactic} approach to Nonstandard Analysis (\cite{wownelly}), as opposed to Robinson's semantic one (\cite{robinson1}), a new predicate `st($x$)', read as `$x$ is standard' is added to the language of \textsf{ZFC}, the usual foundation of mathematics.  
The notations $(\forall^{\st}x)$ and $(\exists^{\st}y)$ are short for $(\forall x)(\st(x)\di \dots)$ and $(\exists y)(\st(y)\wedge \dots)$.  A formula is called \emph{internal} if it does not involve `st', and \emph{external} otherwise.   
The three external axioms \emph{Idealisation}, \emph{Standard Part}, and \emph{Transfer} govern the new predicate `st';  They are respectively defined\footnote{The superscript `fin' in \textsf{(I)} means that $x$ is finite, i.e.\ its number of elements are bounded by a natural number.      } as:  
\begin{enumerate}
\item[\textsf{(I)}] $(\forall^{\st~\textup{fin}}x)(\exists y)(\forall z\in x)\varphi(z,y)\di (\exists y)(\forall^{\st}x)\varphi(x,y)$, for internal $\varphi$ with any (possibly nonstandard) parameters.  
\item[\textsf{(S)}] $(\forall x)(\exists^{\st}y)(\forall^{\st}z)(z\in x\asa z\in y)$.
\item[\textsf{(T)}] $(\forall^{\st}x)\varphi(x, t)\di (\forall x)\varphi(x, t)$, where $\varphi$ is internal,  $t$ is standard and captures \emph{all} parameters of $\varphi$. 
\end{enumerate}
The system \textsf{IST} is (the internal system) \textsf{ZFC} extended with the aforementioned external axioms;  
The former is a conservative extension of \textsf{ZFC} for the internal language, as proved in \cite{wownelly}.    

\medskip

In \cite{brie}, the authors study G\"odel's system $\textsf{T}$ extended with special cases of the external axioms of \textsf{IST}.  
In particular, they consider nonstandard extensions of the (internal) systems \textsf{E-HA}$^{\omega}$ and $\textsf{E-PA}^{\omega}$, respectively \emph{Heyting and Peano arithmetic in all finite types and the axiom of extensionality}.       
We refer to \cite{brie}*{\S2.1} for the exact details of these (mainstream in mathematical logic) systems.  
We do mention that in these systems of higher-order arithmetic, each variable $x^{\rho}$ comes equipped with a superscript denoting its type, which is however often implicit.  
As to the coding of multiple variables, the type $\rho^{*}$ is the type of finite sequences of type $\rho$, a notational device used in \cite{brie} and this paper;  Underlined variables $\underline{x}$ consist of multiple variables of (possibly) different type.  

\medskip

In the next section, we introduce the system $\P$ assuming familiarity with the higher-type framework of G\"odel's system $\textsf{T}$ (See e.g.\ \cite{brie}*{\S2.1} for the latter).    
\subsection{The system $\P$}\label{popo}
In this section, we introduce the system $\P$.  We first discuss some of the external axioms studied in \cite{brie}.  
First of all, Nelson's axiom \emph{Standard part} is weakened to $\HAC_{\INT}$ as follows:
\be\tag{$\HAC_{\INT}$}
(\forall^{\st}x^{\rho})(\exists^{\st}y^{\tau})\varphi(x, y)\di (\exists^{\st}F^{\rho\di \tau^{*}})(\forall^{\st}x^{\rho})(\exists y^{\tau}\in F(x))\varphi(x,y),
\ee
where $\varphi$ is any internal formula.  Note that $F$ only provides a \emph{finite sequence} of witnesses to $(\exists^{\st}y)$, explaining its name \emph{Herbrandized Axiom of Choice}.      
Secondly,  Nelson's axiom idealisation \textsf{I} appears in \cite{brie} as follows:  
\be\tag{\textsf{I}}
(\forall^{\st} x^{\sigma^{*}})(\exists y^{\tau} )(\forall z^{\sigma}\in x)\varphi(z,y)\di (\exists y^{\tau})(\forall^{\st} x^{\sigma})\varphi(x,y), 
\ee
where $\varphi$ is again an internal formula.  
Finally, as in \cite{brie}*{Def.\ 6.1}, we have the following definition.
\bdefi\label{debs}
The set $\T^{*}$ is defined as the collection of all the constants in the language of $\textsf{E-PA}^{\omega*}$.  
The system $ \textsf{E-PA}^{\omega*}_{\st} $ is defined as $ \textsf{E-PA}^{\omega{*}} + \T^{*}_{\st} + \textsf{IA}^{\st}$, where $\T^{*}_{\st}$
consists of the following axiom schemas.
\begin{enumerate}
\item The schema\footnote{The language of $\textsf{E-PA}_{\st}^{\omega*}$ contains a symbol $\st_{\sigma}$ for each finite type $\sigma$, but the subscript is always omitted.  Hence $\T^{*}_{\st}$ is an \emph{axiom schema} and not an axiom.\label{omit}} $\st(x)\wedge x=y\di\st(y)$,
\item The schema providing for each closed term $t\in \T^{*}$ the axiom $\st(t)$.
\item The schema $\st(f)\wedge \st(x)\di \st(f(x))$.
\end{enumerate}
The external induction axiom \textsf{IA}$^{\st}$ is as follows.  
\be\tag{\textsf{IA}$^{\st}$}
\Phi(0)\wedge(\forall^{\st}n^{0})(\Phi(n) \di\Phi(n+1))\di(\forall^{\st}n^{0})\Phi(n).     
\ee
\edefi
For the full system $\P\equiv \textsf{E-PA}^{\omega*}_{\st} +\HAC_{\INT} +\textsf{I}$, we have the following theorem.  
Here, the superscript `$S_{\st}$' is the syntactic translation defined in \cite{brie}*{Def.\ 7.1}, and also listed starting with \eqref{dombu} in the proof of Corollary \ref{consresultcor}.    
\begin{thm}\label{consresult}
Let $\Phi(\tup a)$ be a formula in the language of \textup{\textsf{E-PA}}$^{\omega*}_{\st}$ and suppose $\Phi(\tup a)^\Sh\equiv\forallst \tup x \, \existsst \tup y \, \varphi(\tup x, \tup y, \tup a)$. If $\Delta_{\intern}$ is a collection of internal formulas and
\be\label{antecedn}
\P + \Delta_{\intern} \vdash \Phi(\tup a), 
\ee
then one can extract from the proof a sequence of closed terms $t$ in $\mathcal{T}^{*}$ such that
\be\label{consequalty}
\textup{\textsf{E-PA}}^{\omega*} + \Delta_{\intern} \vdash\  \forall \tup x \, \exists \tup y\in \tup t(\tup x)\ \varphi(\tup x,\tup y, \tup a).
\ee
\end{thm}
\begin{proof}
Immediate by \cite{brie}*{Theorem 7.7}.  
\end{proof}
It is important to note that the proof of the soundness theorem in \cite{brie}*{\S7} provides a \emph{term extraction algorithm} $\mathcal{A}$ to obtain the term $t$ from the theorem.  

\medskip

The following corollary is essential to our results.  We shall refer to formulas of the form $(\forall^{\st}\underline{x})(\exists^{\st}\underline{y})\psi(\underline{x},\underline{y}, \underline{a})$ for internal $\psi$ as (being in) \emph{the normal form}.    
\begin{cor}\label{consresultcor}
If for internal $\psi$ the formula $\Phi(\underline{a})\equiv(\forall^{\st}\underline{x})(\exists^{\st}\underline{y})\psi(\underline{x},\underline{y}, \underline{a})$ satisfies \eqref{antecedn}, then 
$(\forall \underline{x})(\exists \underline{y}\in t(\underline{x}))\psi(\underline{x},\underline{y},\underline{a})$ is proved in the corresponding formula \eqref{consequalty}.  
\end{cor}
\begin{proof}
Clearly, if for $\psi$ and $\Phi$ as given we have $\Phi(\underline{a})^{S_{\st}}\equiv \Phi(\underline{a})$, then the corollary follows immediately from the theorem.  
A tedious but straightforward verification using the clauses (i)-(v) in \cite{brie}*{Def.\ 7.1} establishes that indeed $\Phi(\underline{a})^{S_{\st}}\equiv \Phi(\underline{a})$.  
For completeness, we now list these five inductive clauses and perform this verification.  

\medskip

Hence, if $\Phi(\underline{a})$ and $\Psi(\underline{b})$  in the language of $\P$ have the following interpretations
\be\label{dombu}
\Phi(\underline{a})^{S_{\st}}\equiv (\forall^{\st}\underline{x})(\exists^{\st}\underline{y})\varphi(\underline{x},\underline{y},\underline{a}) \textup{ and } \Psi(\underline{b})^{S_{\st}}\equiv (\forall^{\st}\underline{u})(\exists^{\st}\underline{v})\psi(\underline{u},\underline{v},\underline{b}),
\ee
then they interact as follows with the logical connectives by \cite{brie}*{Def.\ 7.1}:
\begin{enumerate}[(i)]
\item $\psi^{S_{\st}}:=\psi$ for atomic internal $\psi$.  
\item$ \big(\st(z)\big)^{S_{\st}}:=(\exists^{\st}x)(z=x)$.
\item $(\neg \Phi)^{S_{\st}}:=(\forall^{\st} \underline{Y})(\exists^{\st}\underline{x})(\forall \underline{y}\in \underline{Y}[\underline{x}])\neg\varphi(\underline{x},\underline{y},\underline{a})$.  
\item$(\Phi\vee \Psi)^{S_{\st}}:=(\forall^{\st}\underline{x},\underline{u})(\exists^{\st}\underline{y}, \underline{v})[\varphi(\underline{x},\underline{y},\underline{a})\vee \psi(\underline{u},\underline{v},\underline{b})]$
\item $\big( (\forall z)\Phi \big)^{S_{\st}}:=(\forall^{\st}\underline{x})(\exists^{\st}\underline{y})(\forall z)(\exists \underline{y}'\in \underline{y})\varphi(\underline{x},\underline{y}',z)$
\end{enumerate}
Hence, fix $\Phi_{0}(\underline{a})\equiv(\forall^{\st}\underline{x})(\exists^{\st}\underline{y})\psi_{0}(\underline{x},\underline{y}, \underline{a})$ with internal $\psi_{0}$, and note that $\phi^{S_{\st}}\equiv\phi$ for any internal formula.  
We have $[\st(\underline{y})]^{S_{\st}}\equiv (\exists^{\st} \underline{w})(\underline{w}=\underline{y})$ and also 
\[
[\neg\st(\underline{y})]^{S_{\st}}\equiv (\forall^{\st} \underline{W} ) (\exists^{\st}\underline{x})(\forall \underline{w}\in \underline{W}[\underline{x}])\neg(\underline{w}=\underline{y})\equiv (\forall^{\st}\underline{w})(\underline{w}\ne \underline{y}).  
\]    
Hence, $[\neg\st(\underline{y})\vee\neg \psi_{0}(\underline{x}, \underline{y}, \underline{a})]^{S_{\st}}$ is just $(\forall^{\st}\underline{w})[(\underline{w}\ne \underline{y}) \vee \neg \psi_{0}(\underline{x}, \underline{y}, \underline{a})]$, and 
\[
\big[(\forall \underline{y})[\neg\st(\underline{y})\vee \neg\psi_{0}(\underline{x}, \underline{y}, \underline{a})]\big]^{S_{\st}}\equiv
 (\forall^{\st}\underline{w})(\exists^{\st}\underline{v})(\forall \underline{y})(\exists \underline{v}'\in \underline{v})[\underline{w}\ne\underline{y}\vee \neg\psi_{0}(\underline{x}, \underline{y}, \underline{a})].
\]
which is just $(\forall^{\st}\underline{w})(\forall \underline{y})[(\underline{w}\ne \underline{y}) \vee \neg\psi_{0}(\underline{x}, \underline{y}, \underline{a})]$.  Furthermore, we have
\begin{align*}
\big[(\exists^{\st}y)\psi_{0}(\underline{x}, \tup y, \tup a)\big]^{S_{\st}}&\equiv\big[\neg(\forall \underline{y})[\neg\st(\underline{y})\vee\neg \psi_{0}(\underline{x}, \underline{y}, \underline{a})]\big]^{S_{\st}}\\
&\equiv(\forall^{\st} \underline{V})(\exists^{\st}\underline{w})(\forall \underline{v}\in \underline{V}[\underline{w}])\neg[(\forall \underline{y})[(\underline{w}\ne \underline{y}) \vee \neg\psi_{0}(\underline{x}, \underline{y}, \underline{a})]].\\
&\equiv (\exists^{\st}\underline{w})(\exists \underline{y})[(\underline{w}= \underline{y}) \wedge \psi_{0}(\underline{x}, \underline{y}, \underline{a})]]\equiv (\exists^{\st}\underline{w})\psi_{0}(\underline{x}, \underline{w}, \underline{a}).
\end{align*}
Hence, we have proved so far that $(\exists^{\st}\underline{y})\psi_{0}(\underline{x}, \underline{y}, \underline{a})$ is invariant under $S_{\st}$.  By the previous, we also obtain:  
\[
\big[\neg \st(\underline{x})\vee (\exists^{\st}y)\psi_{0}(\underline{x}, \tup y, \tup a)\big]^{S_{\st}}\equiv  (\forall^{\st}\underline{w}')(\exists^{\st} \underline{w})[(\underline{w}'\ne \underline{x}) \vee \psi_{0}(\tup x, \tup w, \tup a)].
\]
Our final computation now yields the desired result: 
\begin{align*}
\big[(\forall^{\st} \underline{x})(\exists^{\st}y)\psi_{0}(\underline{x}, \tup y, \tup a)\big]^{S_{\st}}
&\equiv\big[(\forall \underline{x})(\neg \st(\underline{x})\vee (\exists^{\st}y)\psi_{0}(\underline{x}, \tup y, \tup a))\big]^{S_{\st}}\\
&\equiv(\forall^{\st}\underline{w}')(\exists^{\st} \underline{w})(\forall \underline{x})(\exists \underline{w}''\in \underline{w})[(\underline{w}'\ne \underline{x}) \vee \psi_{0}(\tup x, \tup w'', \tup a)].\\
&\equiv(\forall^{\st}\underline{w}')(\exists^{\st} \underline{w})(\exists \underline{w}''\in \underline{w}) \psi_{0}(\tup w', \tup w'', \tup a).
\end{align*}
The last step is obtained by taking $\underline{x}=\underline{w}'$.  Hence, we may conclude that the normal form $(\forall^{\st} \underline{x})(\exists^{\st}y)\psi_{0}(\underline{x}, \tup y, \tup a)$ is invariant under $S_{\st}$, and we are done.    
\end{proof}
Finally, the previous theorems do not really depend on the presence of full Peano arithmetic.  
Indeed, let \textsf{E-PRA}$^{\omega}$ be the system defined in \cite{kohlenbach2}*{\S2} and let \textsf{E-PRA}$^{\omega*}$ be its extension with types for finite sequences as in \cite{brie}*{\S2}.  
\begin{cor}\label{consresultcor2}
The previous theorem and corollary go through for $\P$ replaced by $\P_{0}\equiv \textsf{\textup{E-PRA}}^{\omega*}+\T_{\st}^{*} +\HAC_{\INT} +\textsf{\textup{I}}$.  
\end{cor}
\begin{proof}
The proof of \cite{brie}*{Theorem 7.7} goes through for any fragment of \textsf{E-PA}$^{\omega{*}}$ which includes \textsf{EFA}, sometimes also called $\textsf{I}\Delta_{0}+\textsf{EXP}$.  
In particular, the exponential function is (all what is) required to `easily' manipulate finite sequences.    
\end{proof}
Finally, we define $\RCAO$ as the system $\P_{0}+\QFAC^{1,0}$.  Recall that Kohlenbach defines $\RCAo$ in \cite{kohlenbach2}*{\S2} as \textsf{E-PRA}$^{\omega}+\QFAC^{1,0}$ 
where the latter is the axiom of choice limited to formulas $(\forall f^{1})(\exists n^{0})\varphi_{0}(f, n)$, $\varphi_{0}$ quantifier-free .     

\subsection{Notations and remarks}
We introduce some notations regarding $\RCAO$.  First of all, we shall follow Nelson's notations as in \cite{bennosam}, and given as follows.
\begin{rem}[Standardness]\label{notawin}\rm
As suggested above, we write $(\forall^{\st}x^{\tau})\Phi(x^{\tau})$ and also $(\exists^{\st}x^{\sigma})\Psi(x^{\sigma})$ as short for 
$(\forall x^{\tau})\big[\st(x^{\tau})\di \Phi(x^{\tau})\big]$ and $(\exists x^{\sigma})\big[\st(x^{\sigma})\wedge \Psi(x^{\sigma})\big]$.     
We also write $(\forall x^{0}\in \Omega)\Phi(x^{0})$ and $(\exists x^{0}\in \Omega)\Psi(x^{0})$ as short for 
$(\forall x^{0})\big[\neg\st(x^{0})\di \Phi(x^{0})\big]$ and $(\exists x^{0})\big[\neg\st(x^{0})\wedge \Psi(x^{0})\big]$.  Furthermore, if $\neg\st(x^{0})$ (resp.\ $\st(x^{0})$), we also say that $x^{0}$ is `infinite' (resp.\ finite) and write `$x^{0}\in \Omega$'.  
Finally, a formula $A$ is `internal' if it does not involve $\st$, and $A^{\st}$ is defined from $A$ by appending `st' to all quantifiers (except bounded number quantifiers).    
\end{rem}
Secondly, we shall use the usual notations for rational and real numbers and functions as introduced in \cite{kohlenbach2}*{p.\ 288-289} (and \cite{simpson2}*{I.8.1} for the former).  
\begin{rem}[Real number]\label{forrealsyo}\rm
A (standard) real number $x$ is a (standard) fast-converging Cauchy sequence $q_{(\cdot)}^{1}$, i.e.\ $(\forall n^{0}, i^{0})(|q_{n}-q_{n+i})|<_{0} \frac{1}{2^{n}})$.  
We freely make use of Kohlenbach's `hat function' from \cite{kohlenbach2}*{p.\ 289} to guarantee that every sequence $f^{1}$ can be viewed as a real.  
Two reals $x, y$ represented by $q_{(\cdot)}$ and $r_{(\cdot)}$ are \emph{equal}, denoted $x=y$, if $(\forall n)(|q_{n}-r_{n}|\leq \frac{1}{2^{n}})$. Inequality $<$ is defined similarly.         
We also write $x\approx y$ if $(\forall^{\st} n)(|q_{n}-r_{n}|\leq \frac{1}{2^{n}})$ and $x\gg y$ if $x>y\wedge x\not\approx y$.  Functions $F$ mapping reals to reals are represented by functionals $\Phi^{1\di 1}$ such that $(\forall x, y)(x=y\di \Phi(x)=\Phi(y))$, i.e.\ equal reals are mapped to equal reals.  Finally, sets are denoted $X^{1}, Y^{1}, Z^{1}, \dots$ and are given by their characteristic functions $f^{1}_{X}$, i.e.\ $(\forall x^{0})[x\in X\asa f_{X}(x)=1]$, where $f_{X}^{1}$ is assumed to be binary.      
\end{rem}
Finally, the notion of equality in $\RCAO$ is important to our enterprise.  
\begin{rem}[Equality]\label{equ}\rm
The system $\RCAo$ includes equality between natural numbers `$=_{0}$' as a primitive.  Equality `$=_{\tau}$' for type $\tau$-objects $x,y$ is defined as follows:
\be\label{aparth}
[x=_{\tau}y] \equiv (\forall z_{1}^{\tau_{1}}\dots z_{k}^{\tau_{k}})[xz_{1}\dots z_{k}=_{0}yz_{1}\dots z_{k}]
\ee
if the type $\tau$ is composed as $\tau\equiv(\tau_{1}\di \dots\di \tau_{k}\di 0)$.
In the spirit of Nonstandard Analysis, we define `approximate equality $\approx_{\tau}$' as follows:
\be\label{aparth2}
[x\approx_{\tau}y] \equiv (\forall^{\st} z_{1}^{\tau_{1}}\dots z_{k}^{\tau_{k}})[xz_{1}\dots z_{k}=_{0}yz_{1}\dots z_{k}]
\ee
with the type $\tau$ as above.  
Furthermore, the system $\RCAo$ includes the axiom of extensionality as follows:
\be\label{EXT}\tag{\textsf{E}}  
(\forall \varphi^{\rho\di \tau})(\forall  x^{\rho},y^{\rho}) \big[x=_{\rho} y \di \varphi(x)=_{\tau}\varphi(y)   \big].
\ee
However, as noted in \cite{brie}*{p.\ 1973}, the axiom of \emph{standard extensionality} \eqref{EXT}$^{\st}$ cannot be included in the system $\P$ (and hence $\RCAO$).  
Finally, a functional $\Xi^{ 1\di 0}$ is called an \emph{extensionality functional} for $\varphi^{1\di 1}$ if 
\be\label{turki}
(\forall k^{0}, f^{1}, g^{1})\big[ \overline{f}\Xi(f,g, k)=_{0}\overline{g}\Xi(f,g,k) \di \overline{\varphi(f)}k=_{0}\overline{\varphi(g)}k \big].  
\ee
In other words, $\Xi$ witnesses \eqref{EXT} for $\Phi$.  As will become clear in Section \ref{tempie}, standard extensionality is translated by our algorithm $\RS$ into the existence of an extensionality functional, and the latter amounts to merely an unbounded search.   
\end{rem}

\section{Classifying \UDNR}\label{CUDNR}
%
In this section, we prove that the principle $\UDNR$ from the introduction is equivalent to arithmetical comprehension $(\exists^{2})$ as follows:
\be\tag{$\exists^{2}$}
(\exists \varphi^{2})(\forall g^{1})\big[(\exists x^{0})g(x)=0 \asa \varphi(g)=0  \big].
\ee
We shall even establish an \emph{explicit} equivalence between $\UDNR$ and a version of $(\exists^{2})$.
\bdefi[Explicit implication]
An implication $(\exists \Phi)A(\Phi)\di (\exists \Psi)B(\Psi)$ is \emph{explicit} if there is a term $t$ in the language such that additionally $(\forall \Phi)[A(\Phi)\di B(t(\Phi))]$, i.e.\ $\Psi$ can be explicitly defined in terms of $\Phi$.  
\edefi
To establish the aforementioned explicit equivalence, we shall obtain a suitable \emph{nonstandard} equivalence in $\RCAO$, and apply Corollary \ref{consresultcor}.  
We first prove the following theorem, where $\UDNR^{+}$ is 
\[
(\exists^{\st}\Psi^{1\di 1})\big[(\forall^{\st} A^{1})(\forall e^{0})(\Psi(A)(e)\ne \Phi_{e}^{A}(e))\wedge (\forall^{\st} C^{1}, D^{1})\big(C\approx_{1} D \di \Psi(C)\approx_{1}\Psi(D) \big)\big].  
\]
Note that the second conjunct expresses that $\Psi$ is \emph{standard extensional} (See Remark~\ref{equ}).  
We also need the following restriction of Nelson's axiom \emph{Transfer}:
\be\tag{$\paai$}
(\forall^{\st}f^{1})\big[(\forall^{\st}n^{0})f(n)=0\di (\forall m)f(m)=0\big].
\ee  
\begin{thm}\label{UDNRPLUS}
In $\RCAO$, we have $\UDNR^{+}\asa \paai$.  
\end{thm}
\begin{proof}
To prove $\paai \di \textup{\UDNR}^{+}$, define:  
\be\label{kukkkk}
\Theta(A,M)(e):=
\begin{cases}
 \Phi_{e,M}^{A}(e)+1& (\exists y,s\leq M)(\Phi_{e,s}^{A}(e)=y) \\
0 & \textup{otherwise}
\end{cases}.
\ee
Assuming $\paai$, the functional from \eqref{kukkkk} clearly satisfies:
\be\label{frik2}
(\forall^{\st}e^{0}, A^{1})(\forall M,N\in \Omega)\big[\Theta(A,M)(e)=\Theta(A,N)(e)\big].
\ee
The formula \eqref{frik2} clearly implies 
\be\label{frik}
(\forall^{\st}e^{0}, A^{1})(\exists k^{0})(\forall M,N\geq k)\big[\Theta(A,M)(e)=\Theta(A,N)(e)\big].
\ee
Since $\RCAO$ proves minimisation for $\Pi_{1}^{0}$-formulas, there is a least $k$ as in \eqref{frik}, which must be finite by \eqref{frik2}.  Hence, we obtain:
\be\label{frik3}
(\forall^{\st}e^{0}, A^{1})(\exists^{\st} k^{0})(\forall M,N\geq k)\big[\Theta(A,M)(e)=\Theta(A,N)(e)\big].
\ee
Applying \HAC$_{\INT}$, there is a standard functional $\Psi^{2}$ such that 
\be\label{frik4}
(\forall^{\st}e^{0}, A^{1})(\exists k^{0}\in \Psi(A, e))(\forall M,N\geq k)\big[\Theta(A,M)(e)=\Theta(A,N)(e)\big].
\ee
Now define $\Xi(A)(e)$ as $\Theta(A, \zeta(A, e))(e)$, where $\zeta(A, e)$ is the maximum of $\Psi(A, e)(i)$ for $i<|\Psi(A, e)|$.  We then have that:
\be\label{kukkkkl}
(\forall^{\st}e^{0}, A^{1})(\forall M\in \Omega)\big[\Theta(A,M)(e)=\Xi(A)(e)\big].
\ee
By definition of $\Theta$ in \eqref{kukkkk}, $\Xi$ is standard extensional (which follows from applying $\paai$ to the associated axiom of extensionality) 
and satisfies, for standard $A$, the formula $(\forall^{\st} e^{0})\big[\Xi(A)(e)\ne \Phi_{e}^{A}(e)\big]$, where the `st' predicates in the latter formula may be dropped by $\paai$.  
Hence, $\Xi$ is as required for \UDNR$^{+}$.  

\medskip

We now prove ${\UDNR}^{+}\di \paai$.  To this end, assume the former and suppose the latter is false, i.e.\ there is standard $h^{1}$ such that $(\forall^{\st}n)h(n)=0$ and $(\exists m)h(m)\ne0$.   
Next, fix a standard pairing function $\pi^{1}$ and its inverse $\xi^{1}$.  Now let the \emph{standard} number $e_{1}$ be the code of the following program:  On input $n$, set $k=n$ and check if $k\in A$ and if so, return the second component of $\xi(k)$;  If $k\not\in A$, repeat for $k+1$.  Intuitively speaking, $e_{1}$ is such that $\Phi_{e_{1}}^{A}(n)$ outputs $m$ if starting at $k=n$, we eventually find $\pi((l,m))\in A$, and undefined otherwise. 
Furthermore, define $C=\emptyset$ (which is the sequence $00\dots$) and 
\[
D=\{\pi(e, \Psi(C)(e_{1})): h(e)\ne 0\wedge (\forall i<e)h(i)= 0\}, 
\]   
where $h$ is the exception to $\paai$.  
Note that $C\approx_{1}D$ by definition, implying that $\Psi$ satisfies $\Psi(C)\approx_{1} \Psi(D)$ due to its standard extensionality.  
However, the latter combined with $\UDNR$ gives us:
\be\label{kilopol}
\Psi(C)(e_{1})=_{0}\Psi(D)(e_{1})\ne_{0} \Phi_{e_{1}, m_{0}}^{D}(e_{1})=_{0}\Psi(C)(e_{1}),   
\ee
for large enough (infinite) $m_{0}$.
This  contradiction yields the theorem.  
\end{proof}
For the following theorem, we require Feferman's mu-operator:
\be\label{mu}\tag{$\mu^{2}$}
(\exists \mu^{2})\big[(\forall f^{1})( (\exists n)f(n)=0 \di f(\mu(f))=0)    \big], 
\ee
which is equivalent to $(\exists^{2})$ over $\RCAo$ by \cite{kohlenbach2}*{Prop.\ 3.9}.   As to notation, denote by $\textsf{MU}(\mu)$ the formula in square brackets in \eqref{mu} and denote by $\UDNR(\Psi)$ the formula in square brackets in \textsf{UDNR}.  
We have the following theorem.  
\begin{thm}\label{sef}
From the proof of $\UDNR^{+}\asa \paai$ in $\RCAO $, two terms $s, u$ can be extracted such that $\RCAo$ proves:
\be\label{frood}
(\forall \mu^{2})\big[\textsf{\MU}(\mu)\di \UDNR(s(\mu)) \big] \wedge (\forall \Psi^{1\di 1})\big[ \UDNR(\Psi)\di  \MU(u(\Psi, \Phi))  \big],
\ee
where $\Phi$ is an extensionality functional for $\Psi$
\end{thm}
\begin{proof}
We prove the second conjunct in \eqref{frood};  The first conjunct is analogous.  We first show that $\UDNR^{+}\di \paai$ can be brought in the normal form from Corollary \ref{consresultcor}.  
First of all, note that $\paai$ is easily brought into the form:
\be\label{dirf}
(\forall^{\st} f^{1})(\exists^{\st} y^{0})\big[(\exists x^{0})f(x)=0 \di (\exists z^{0}\leq y)f(z)=0\big].
\ee
In $\UDNR^{+}$, resolve the predicates `$\approx_{1}$' in the second conjunct to obtain:
\[
 (\forall^{\st}X^{1}, Y^{1}, k^{0})(\exists^{\st}N^{0})(\overline{X}N=_{0}\overline{Y}N\di \overline{\Psi(X)}k=_{0} \overline{\Psi(Y)}k).
\]
Apply $\HAC_{\INT}$ to obtain standard $\Phi$ such that $(\exists N\in \Phi(X, Y, k))$.  Define $\Xi(X, Y, k)$ as $\max_{i<|\Phi(X, Y, k)|}\Phi(X, Y, k)(i)$ to obtain  
\[
 (\exists^{\st}\Xi)(\forall^{\st}X^{1}, Y^{1}, k^{0})\big[\overline{X}\Xi(X, Y, k)=_{0}\overline{Y}\Xi(X, Y, k)\di \overline{\Psi(X)}k=_{0} \overline{\Psi(Y)}k\big].
\]
Let $B(\Xi, X, Y, k)$ be the formula in square brackets in the previous formula and let $C(f, y)$ be the formula in square brackets in \eqref{dirf}.  So far, we have derived
\[
\big[(\exists^{\st}\Psi)(\forall^{\st}Z^{1})A(Z, \Psi)\wedge  (\exists^{\st}\Xi)(\forall^{\st}X^{1}, Y^{1}, k^{0})B(\Xi, X, Y, k)\big] \di (\forall^{\st}f^{1})(\exists^{\st}y^{0})C(f, y),
\]  
from $\UDNR^{+}\di \paai$ in $\RCA_{0}^{\Lambda}$, where $A(Z, \Psi)\equiv (\forall e^{0})(\Psi(Z)(e)\ne \Phi_{e}^{Z}(e))$.  By bringing outside all the standard quantifiers, we obtain  
\be\label{dth}
(\forall^{\st}f, \Psi, \Xi)(\exists^{\st}y^{0}, X^{1},Y^{1}, Z^{1}, k )\big[[A(Z, \Psi)\wedge B(\Xi, X, Y, k)] \di C(f, y)\big],
\ee
where the formula in square brackets is internal.  
Thanks to Corollary~\ref{consresultcor}, the term extraction algorithm $\mathcal{A}$ applied to `$\RCAO\vdash \eqref{dth}$', provides a term $t$ such that
\[
(\forall f, \Psi, \Xi)(\exists (y^{0}, X^{1},Y^{1}, Z^{1}, k)\in t(f, \Psi, \Xi) )\big[[A(Z, \Psi)\wedge B(\Xi, X, Y, k)] \di C(f, y)\big]
\]  
is provable in $\RCAo$.  
Now let $s$ be the term $t$ with all entries not pertaining to $y$ omitted;  We have
\[
(\forall f, \Psi, \Xi)(\exists k^{0}, X^{1},Y^{1}, Z^{1})(\exists y\in s(f, \Psi, \Xi) )\big[[A(Z, \Psi)\wedge B(\Xi, X, Y, k)] \di C(f, y)\big].  
\]  
Now define $u(f, \Psi, \Xi)$ as the maximum of all entries of $s(f, \Psi, \Xi)$;  We have
\[
(\forall f, \Psi, \Xi)(\exists k^{0}, X^{1},Y^{1}, Z^{1})\big[[A(Z, \Psi)\wedge B(\Xi, X, Y, k)] \di C(f, u(f, \Psi, \Xi))\big].  
\]  
Bringing all quantifiers inside again as far as possible, we obtain
\[
(\forall  \Psi, \Xi)\big[[(\forall Z^{1})A(Z, \Psi)\wedge (\forall k^{0}, X^{1},Y^{1})B(\Xi, X, Y, k)] \di (\forall f)C(f, u(f, \Psi, \Xi))\big].  
\]  
Hence, if $\Psi_{0}$ is as in $\UDNR$ and $\Xi_{0}$ witnesses the extensionality of $\Psi_{0}$, then $u(\cdot, \Psi_{0}, \Xi_{0})$ is Feferman's my-operator, and we are done.    
\end{proof}
\begin{cor}\label{essek}
In $\RCAo$, we have the explicit\footnote{Since the system $\RCAo$ includes the axiom of extensionality \eqref{EXT} and $\QFAC^{1,0}$, and in 
light of the elementary nature (an unbounded search) of an extensionality functional $\Xi$, we will still call $t(\Xi, \cdot)$ `explicit' if $t$ is a term from the language.  } equivalence $\UDNR\asa (\mu^{2})$.
\end{cor}
Clearly, there is a general strategy to obtain the normal form as in \eqref{dth} for principles similar to $\UDNR^{+}$, as discussed in the following remark.  
\begin{rem}[Algorithm $\mathcal{B}$]\label{algob}\rm
Let $T\equiv (\forall X^{1})(\exists Y^{1})\varphi(X, Y)$ be an internal formula and define the `strong' uniform version $UT^{+}$ as 
\[
(\exists^{\st} \Phi^{1\di 1})\big[(\forall^{\st} X^{1})\varphi(X,\Phi(X)) \wedge \textup{$\Phi$ is standard extensional}\big].
\]
The proof of Theorem \ref{sef} provides a \emph{normal form algorithm} $\mathcal{B}$ to convert the implication $UT^{+}\di \paai$ into the normal form $(\forall^{\st}x)(\exists^{\st}y)\varphi(x, y)$ as in \eqref{dth}. 
\end{rem}
In general, if $ T\di \DNR$ and the proof of the implication is sufficiently uniform, then $UT\di (\exists^{2})$ and this implication is explicit.  We now list some examples.  
\begin{rem}[Immediate consequences]\rm
First of all, let DNR$_{k}$ be DNR where the function $f^{1}$ satisfies $f\leq_{1}k$, and let UDNR$_{k}$ be UDNR with the same restriction on $\Psi(A)$.  
Clearly, for any $k\geq 1$, we have the explicit equivalence $\UDNR_{k}\asa (\exists^{2})$.

\medskip

Secondly, let \textsf{RKL} be the `Ramsey type' version of $\WKL$ from \cite{stoptheflood} and let \textsf{URKL} be its obvious uniform version.  
In $\RCAo$, we have the explicit equivalence $\textup{\textsf{URKL}}\asa (\exists^{2})$, as it seems the proof of \textsf{RKL} $\di \DNR$ from \cite{stoptheflood}*{Theorem 8} can be uniformized.  
Indeed, in this proof, \textsf{RKL} is applied to a specific tree $T_{0}$ from \cite{stoptheflood}*{Lemma 7} to obtain a certain set $H$.   Then the function $g$ is defined such that $W_{g(e)}$ is the least $e+3$ elements of $H$.  
This function $g$ is then shown to be fixed-point free, which means it gives rise to a \DNR-function by \cite{zweer}*{V.5.8, p.\ 90}.  
Noting that the tree $T_{0}$ has positive measure, we even obtain \textsf{WRKL} $\di \DNR$ (and the associated uniform equivalence to $(\exists^{2})$), where the tree has positive measure in the latter (See \cite{paul1}).  

\medskip

Thirdly, let \textsf{SEM} be the \emph{stable} Erd\"os-Moser theorem from \cite{lemans}.  In \cite{patey1}*{Theorem~3.11}, the implication \textsf{SEM} $\di \DNR$ is proved, and the proof is clearly uniform.    
Hence, for \textsf{USEM} the uniform version of \textsf{SEM}, we have (explicitly) that \textsf{USEM} $\asa (\exists^{2})$.  The same obviously holds for \textsf{EM}, the version of \textsf{SEM} without stability.  
\end{rem}
In the next section, we shall study principles from the zoo for which a `uniformising' proof as in the previous remark is not immediately available.  
We finish this section with a remark on the Reverse Mathematics zoo.  
\begin{rem}[A higher-order zoo]\label{dazoo}\rm
Since $\DNR$ is rather `low' in the zoo, it is to be expected that uniform versions of `most' of the zoo's principles will behave as \UDNR, i.e.\ turn out equivalent to $(\exists^{2})$ (as we will establish below).  
In particular, since Friedman-Simpson style Reverse Mathematics is limited to second-order arithmetic, the proof of Theorem~\ref{UDNRPLUS} will go through for principles other than $\UDNR$ as the associated functionals can 
only have type $1\di 1$ (by the limitation to second-order arithmetic).  However, it is conceivable that uniform \emph{higher-type} principles, to which the proof of Theorem~\ref{UDNRPLUS} does not apply, will populate a `higher-order' RM zoo.
\end{rem}

\section{Classifying the Reverse Mathematics zoo}\label{czoo}
In this section, we classify uniform versions of a number principles from the RM zoo, based on the results in the previous section.  
After these case studies, we shall formulate in Section \ref{tempie} a template which seems sufficiently general to apply 
to virtually any (past or future) principle from the RM zoo.  

\subsection{Ascending and descending sequences}\label{padis}
In this section, we study the uniform version of the ascending-descending principle $\ADS$ (See e.g.\ \cite{dslice}*{Def.~9.1}).  
\bdefi
For a linear order $\preceq$, a sequence $x_{n}^{1}$ is \emph{ascending} if $x_{0}\prec x_{1}\prec \dots$ and \emph{descending} if $x_{0}\succ x_{1}\succ \dots$.
\edefi
\begin{defi}[\ADS]
Every infinite linear ordering has an ascending or a descending sequence.
\edefi
Recall that $\textup{LO}(X^{1})$ is short for `$X^{1}$ is a linear order'; We append `$\infty$' to `LO' to stress that $X$ is an infinite\footnote{Here, `infinite' should not be confused with the notation `$M^{0}$ is infinite' for $\neg\st(M)$; Note the type mismatch between numbers and orders.} linear order, meaning that its field is not bounded by any number (See \cite{simpson2}*{V.1.1}).  
With this in place, uniform $\ADS$ is as follows:
\begin{defi}[\UADS]
\begin{align}
(\exists \Psi^{1\di 1})(\forall X^{1})\big[\textup{LO}_{\infty}(X)& \di (\forall n^{0})\Psi(X)(n)<_{X} \Psi(X)(n+1)\notag \\ 
&\vee (\forall m^{0})\Psi(X)(m)>_{X} \Psi(X)(m+1)\big].\label{dardenne}
\end{align}
\edefi
Note that we can decide which case of the disjunction of $\UADS$ holds by testing $\Psi(X)(0)<_{X}\Psi(X)(1)$.  We have the following theorem.  
\begin{thm}\label{JUDAS}
In $\RCAo$, we have the explicit equivalence $\UADS\asa(\mu^{2})$. 
\end{thm}
\begin{proof}
Since all notions involved are arithmetical, the explicit implication $(\mu^{2})\di \UADS$ is straightforward.  
For the remaining explicit implication, we will prove $\UADS^{+}\di  \paai$, where the former is 
\[
(\exists^{\st}\Psi^{1\di 1})\big[(\forall^{\st} X^{1})A(X, \Psi)\wedge (\forall^{\st} X^{1}, Y^{1})\big(X\approx_{1} Y \di \Psi(X)\approx_{1}\Psi(Y) \big)\big],
\]
where $A(X, \Psi)$ is the formula in square brackets in \eqref{dardenne}.  It is then easy to bring the implication $\UADS^{+}\di \paai$ in the normal form as in \eqref{dth} using the algorithm $\mathcal{B}$ from Remark \ref{algob}.  
Applying the term extraction algorithm $\mathcal{A}$ using Corollary \ref{consresultcor} then establishes the explicit implication $\UADS\di (\mu^{2})$.       

\medskip

Thus, assume $\UADS^{+}$ and suppose $\paai$ is false, i.e.\ there is a standard $h^{1}$ such that $(\forall^{\st}n)h(n)=0$ and $(\exists m)h(m)\ne0$.  Now let $m_{0}$ be the least number such that $h(m_{0})\ne 0$ and define the ordering `$\prec$' as follows:
\be\label{heneriek}
\dots \prec m_{0}+2 \prec m_{0}+1\prec 0\prec 1\prec 2\prec \dots \prec m_{0}.
\ee
It is straightforward\footnote{The order $\prec$ from \eqref{heneriek} can be defined as:  $i\prec j$ holds if $i<j\wedge (\forall k\leq j-1)h(k)=0$ or $i>j\wedge (\exists k\leq j-1)h(k)\ne0$ or $i>j\wedge (\exists k\leq j-1)h(k)\ne0\wedge (\forall k\leq i-1)h(k)=0$.\label{lootnote}} to define the \emph{standard} ordering $\prec$ using the function $h$.  Now consider the usual strict ordering $<_{0}$ and note that $(\prec)~ \approx_{1}~(<_{0})$ (with some abuse of notation in light of \cite{simpson2}*{V.1.1}).  
By the standardness of $\prec$ and standard extensionality for the standard $\Psi$ functional from $\UADS^{+}$, we have $\Psi(\prec)\approx_{1}\Psi(<_{0})$ (again with some abuse of notation).  However, this leads to a contradiction as $<_{0}$ only has 
ascending infinite sequences, while $\prec$ only has descending infinite sequences.  Indeed, while only the first case in \eqref{dardenne} can hold for $\Psi(<_{0})$, only the second case can hold for $\Psi(\prec)$.  But then $\Psi(\prec)\approx_{1}\Psi(<_{0})$ is impossible.  
This contradiction guarantees that $\UADS^{+} \di \paai$, and we are done.    
\end{proof}
In \cite{schirfeld}*{Prop.\ 3.7}, it is proved that $\ADS$ is equivalent to the principle \textsf{CCAC}.  In light of the uniformity of the associated proof, the uniform version of the latter is also equivalent to $(\exists^{2})$.  
Furthermore, the equivalence in the previous theorem translates into a result in \emph{constructive} Reverse Mathematics (See \cite{ishi1}) as follows.
\begin{rem}[Constructive Reverse Mathematics]\label{CRM}\rm
The ordering $\prec$ defined in \eqref{heneriek} yields a proof that $\ADS\di \Pi_{1}^{0}\textup{-LEM}$ over the (constructive) base theory from \cite{ishi1}.  Indeed, for a function $h^{1}$, define the ordering $\prec_{h}$ from Footnote \ref{lootnote}.  
By \ADS, there is a sequence $x_{n}$ which is either ascending or descending in $\prec_{h}$.  It is now easy to check that if $x_{0}\prec_{h} x_{1}$, then $(\forall n)h(n)=0$, and if $x_{0}\succ_{h} x_{1}$ then $\neg [(\forall n)h(n)=0]$.        
Hence, ADS provides a way to decide whether a $\Pi_{1}^{0}$-formula holds or not, i.e.\ the law of excluded middle limited to $\Pi_{1}^{0}$-formulas.  
\end{rem}
Next, we consider a special case of \ADS.    
The notion of \emph{discrete} and \emph{stable} linear orders from \cite{dslice}*{Def.\ 9.15} is defined as follows.  
\bdefi[{Discrete and stable orders}]\label{DISQUE}
A linear order is \emph{discrete} if every element has an immediate predecessor, except for the first element of the order if there is one, and every element has an immediate successor, except for the last element of the order if there is one. 
A linear order is \emph{stable} if it is discrete and has more than one element, and every element has either finitely many predecessors or finitely many successors. (Note that a stable order must be infinite.)
\edefi
Again, to be absolutely clear, the notion of `finite' and `infinite' in the previous definition constitutes the `usual' \emph{internal} definitions of infinite orders in $\RCAo$ and have nothing to do with our notation `$M$ is infinite' for $\neg\st(M^{0})$.  
In particular, note the type mismatch between orders and numbers.  

\medskip

Now denote by \textsf{SADS} the principle $\ADS$ limited to stable linear orderings, and let \textsf{USADS} be its uniform version.
\begin{cor}\label{thisone}
In $\RCAo$, we have the explicit equivalence $\textup{\textsf{USADS}}\asa(\mu^{2})$,  
\end{cor}
\begin{proof}
Note that both the orderings $<_{0}$ and $\prec$ defined in the proof of the theorem are stable and this proof thus also yields $\textsf{USADS}^{+}\di \paai$.  
\end{proof}
Let $\textup{\textsf{SRT}}_{2}^{2}$ be Ramsey's theorem for pairs limited to stable colourings (See e.g.\ \cite{dslice}*{Def.\ 6.28}), and let \textsf{USRT}$_{2}^{2}$ be its uniform version
where a functional $\Psi^{1\di 1}$ takes as input a stable 2-colouring of pairs of natural numbers and outputs an infinite homogeneous set.
\begin{cor}  
In $\RCAo$, we have the explicit equivalence $\textup{\textsf{USRT}}_{2}^{2}\asa(\mu^{2})$. 
\end{cor}
\begin{proof}
By \cite{schirfeld}*{Prop.\ 2.8}, we have $\textsf{SRT}_{2}^{2}\di\textsf{SADS}$ .  The proof of the latter is clearly uniform, yielding the forward implication by Corollary \ref{thisone}.  
By \cite{yamayamaharehare}*{Theorem 4.2}, the reverse implication follows.      
\end{proof}
We can prove similar results for \textsf{SRAM} and related principles from \cite{sram}, but do not go into details.  
Our next corollary deals with the \emph{chain-antichain principle}.
\bdefi[\textsf{CAC}]
Every infinite partial order $(P, \leq_{P} )$ has an infinite subset $S$ that is either a \emph{chain}, i.e.\ $(\forall x^{0},y^{0}\in S)(x\leq_{P}\vee y\leq_{P} x)$, or an \emph{antichain}, i.e.\ $(\forall x^{0},y^{0}\in S)(x\ne y\di x\not\leq_{P}\vee y\not\leq_{P} x)$.
\edefi
Let \textsf{UCAC} be the principle \textsf{CAC} with the addition of a functional $\Psi^{1\di 1}$ such that $\Psi(P, \leq_{P})$ is the infinite subset which is either a chain or antichain.  Let \textsf{USCAC} be \textsf{UCAC} limited to \emph{stable} partial orders (See \cite{schirfeld}*{Def.\ 3.2}).  
\begin{cor}
In $\RCAo$, we have the explicit equivalences $\UCAC\asa (\mu^{2})\asa \textup{\textsf{USCAC}}$.
\end{cor}
\begin{proof}
In \cite{schirfeld}*{Prop.\ 3.1}, $\CAC\di \ADS$ is proved.  The proof is clearly uniform, implying the explicit implication $\UCAC\di \UADS$.  
By Theorem \ref{JUDAS}, we obtain the first forward implication in the theorem.  The first reverse implication is proved as in the final part of the proof of Theorem \ref{JUDAS}.  
For the final reverse implication, the implication \textsf{SCAC} $\di$ \textsf{SADS} is proved in \cite{schirfeld}*{Prop.\ 3.3}.  Since the latter proof is clearly uniform, we have (explicitly) that \textsf{USCAC} $\di (\exists^{2})$ by Corollary \ref{thisone}.    
The remaining implication is immediate.  
\end{proof}
Finally, we point out one important feature of the above proofs.  
\begin{rem}[Discontinuities]\label{dikko}\rm
We show that the construction \eqref{heneriek} which gives rise to $\UADS^{+}\di \paai$, also implies the existence of a discontinuity (in the sense of Nonstandard Analysis).  
Indeed, for infinite $M$, define $g^{1}\equiv00\dots 00100\dots$ where $g(M)=1$.  Let the (nonstandard) order $\vartriangleleft$ be the order $\prec$ as in \eqref{heneriek} but with $g$ instead of $h$.   
Then clearly $(<_{0})\approx_{1}(\vartriangleleft)\wedge \Psi(<_{0})\not\approx_{1}\Psi(\vartriangleleft)$ for $\Psi$ as in $\UADS$, i.e.\ this functional is not nonstandard continuity `around' $<_{0}$.  
A similar construction involving $g$ gives rise to a discontinuity around \emph{any} standard input.      
In conclusion, the functional $\Psi$ from $\UADS^{+}$ is `everywhere discontinuous' in the sense of Nonstandard Analysis.  
This observation applies to all the RM zoo principles discussed in this section.  
Thus, principles of the form $(\forall X^{1})(\exists Y^{1})\varphi(X, Y)$ from the RM zoo can be said to be `not continuous in their input parameter $X$'.  
\end{rem}

\subsection{Thin and free sets}
In this section, we study the so-called thin- and free set theorems from \cite{freesets}.  
In the latter, the thin set theorem $\textsf{TS}$ is defined as follows; $\textsf{TS}(k)$ is $\textsf{TS}$ limited to some fixed $k\geq 1$. 
\begin{princ}[$\textsf{TS}$]
$(\forall k)(\forall f:[\N]^{k}\di \N)(\exists A )(A\textup{ is infinite } \wedge f([A]^{k})\ne \N)$.
\end{princ}
We define $\textup{\textsf{UTS}}(2)$ as follows:
\be\tag{$\textup{\textsf{UTS}}(2)$}\label{ttwo}
(\exists \Psi^{1\di 1})(\forall f^{1}:[N]^{2}\di N)\big[\Psi(f) \textup{ is infinite } \wedge (\exists n^{0})\big[n\not\in f\big([\Psi(f)]^{2}\big)\big]  \big].
\ee  
We did not use `$\N$' to avoid confusion.  Recall that `$\Psi(f)$ is infinite' has nothing to do with infinite numbers $M\in \Omega$;  Note in particular the type mismatch.  
\begin{thm}
In $\RCAo$, we have the explicit equivalence $(\mu^{2})\asa \textup{\textsf{UTS}}(2)$.
\end{thm}  
\begin{proof}
The forward (explicit) implication is immediate from the results in \cite{freesets}*{\S5}.
For the reverse implication, we will prove a suitable implication in $\RCAO$ and apply the algorithms $\mathcal{B}$ and $\mathcal{A}$ using Corollary \ref{consresultcor}.  
Hence, let $\Psi$ be as in \textsf{UTS}$(2)$ and apply \textsf{QF-AC}$^{1,0}$ to $(\forall f^{1}:[N]^{2}\di N)(\exists n^{0})\big[n\not\in f\big([\Psi(f)]^{2}\big)\big] $ to obtain $\Xi^{2}$ witnessing $n^{0}$.
In this way, \textsf{UTS}$(2)$ becomes
\[
(\exists \Phi^{1\di (1\times 0)})(\forall f^{1}:[N]^{2}\di N)\big[\Phi(f)(1) \textup{ is infinite } \wedge \Phi(f)(2)\not\in f\big([\Phi(f)(1)]^{2}\big)  \big].
\]
Let $A(\Phi, f)$ be the formula in square brackets  and define $\textsf{UTS}(2)^{+}$ as 
\[
(\exists^{\st} \Phi^{1\di (1\times 0)})(^{\st}\forall f^{1}:[N]^{2}\di N)\big[ A(\Phi, f) \wedge\Phi \textup{ is standard extensional}  \big].
\]
We now prove that $\textsf{UTS}(2)^{+}\di \paai$;
To this end, assume the latter and suppose $h^{1}$ is a counterexample to $\paai$, i.e.\ $(\forall^{\st}n)h(n)=0 \wedge (\exists m)h(m)\ne0$.  
Fix standard $ f^{1}:[N]^{2}\di N$ and define $g^{1}:[N]^{2}\di N$ as:
\be\label{meeh}
g(k,l):=
\begin{cases}
f(k,l) & (\forall i\leq \max(k,l))h(i)=0\\
\Phi(f)(2) &\textup{otherwise}
\end{cases}.
\ee
By assumption, $f\approx_{1}g$, and we obtain $\Phi(f)\approx_{(1\times 0)}\Phi(g)$ by the standard extensionality of $\Phi$.  
Note that in particular $\Phi(f)(2)=\Phi(g)(2)$, and since $\Phi(g)(1)$ is infinite, there are some $k_{0}'>k_{0}>m_{0}$ such that $k_{0},k_{0}'\in \Phi(g)(1)$ where $m_{0}$ is such that $h(m_{0})\ne 0$.    
However, by the definition of $g$, we obtain $\Phi(f)(2)\in g([\Phi(g)(1)]^{2})$, as we are in the second case of \eqref{meeh} for $g(k_{0}, k_{0}')$.     
Since $\Phi(f)(2)=\Phi(g)(2)$, the previous yields the contradiction $\Phi(g)(2)\in g([\Phi(g)(1)]^{2})$, and hence $\paai$ must hold.  
Now bring $\textsf{UTS}(2)^{+}\di \paai$ in the normal form using $\mathcal{B}$ and apply term extraction via $\mathcal{A}$, using Corollary \ref{consresultcor}.    
\end{proof}
Clearly, the previous proof also goes through for the uniform version of \textsf{STS}$(2)$, which is \textsf{TS}$(2)$ limited to \emph{stable} functions, 
i.e.\ for functions $f:[N]^{2}\di N$ such that $(\forall x^{0})(\exists y^{0})(\forall z^{0}\geq y)(f(x,y)=f(x,z))$.

\medskip

Next, we consider the following corollary regarding the free set theorem, where $\textup{\textsf{UTS}}(k)$ and $\textup{\textsf{UFS}}(k)$ have obvious definitions in light of the notations in \cite{freesets}.  
\begin{cor}
In $\RCAo$, we have \(explicitly\) that $(\mu^{2})\asa \textup{\textsf{UTS}}(k)\asa \textup{\textsf{UFS}}(k)$, where $k\geq 1$.
\end{cor}  
\begin{proof}
The case $k\geq 2$ is immediate from the theorem, the uniformity of \cite{freesets}*{Theorems 3.2 and 3.4}, and the fact that $\ACA_{0}$ proves $\textsf{FS}$ (\cite{freesets}).  
To obtain the set $B$ in the proof of the former theorem, apply \textsf{QF-AC}$^{1,0}$ to the fact that the free set is infinite.  For the case $k=1$, proceed as in the theorem.    
\end{proof}
As noted by Kohlenbach in \cite{kohlenbach2}*{\S3}, the (necessary) use of the law of excluded middle in the proof of a theorem, gives rise to a discontinuity in the uniform version of this theorem.  
Now, even the proof of $\textsf{FS}(1)$ in \cite{freesets}*{Theorem 2.2} uses this law, explaining the equivalence to $(\exists^{2})$ of the associated uniform version.

\subsection{Cohesive sets}\label{tempo2}
In this section, we study principles based on \emph{cohesiveness} (See e.g.\ \cite{dslice}*{Def.\ 6.30}).  We start with the principle \COH.    
\bdefi
A set $C$ is cohesive for a collection of sets $R_{0}, R_{1},\dots$  if it is infinite and for each $i$, either $C\subseteq^{*} R_{i}$ or $C\subseteq^{*} \overline{R}_{i}$. 
Here, $\overline{A}$ is the complement of $A$ and $A\subseteq^{*}B$ means that $A\setminus B$ is finite.    
\edefi

\bdefi[\COH]
Every countable collection of sets has a cohesive set.
\edefi
It is important to note that $\COH$ involves multiple significant existential quantifiers:  The `$(\exists C^{1})$' quantifier, but also the existential type 0-quantifiers in $C\subseteq^{*} R_{i}\vee C\subseteq^{*} \overline{R}_{i}$.   
As we will see, it is important that the functional from the uniform version of $\COH$ outputs both the set $C$ \emph{and} an upper bound to $C\setminus R_{i}$ or $C\setminus \overline{R}_{i}$.
It would be interesting, but beyond the scope of this paper, to study a weak version of $\UCOH$ only outputting $C$.  
\bdefi[\UCOH] There is $\Phi^{(0\di 1)\di (1\times 1)}$ such that for all $ R^{0\di 1}$
\begin{align}
(\forall k^{0})(\exists l^{0}>k)[l\in & \Phi(R)(1)] \wedge   (\forall i^{0})\Big[\big(\forall n\in \Phi(R)(1)\big)(n\geq \Phi(R)(2)(i)\di n\in R(i)) \notag\\
& \vee\big(\forall m\in \Phi(R)(1)\big)(m\geq \Phi(R)(2)(i)\di m\in \overline{R(i)})\Big].\label{dfghj}
\end{align}
\edefi
Note that we may treat the collection $R^{0\di 1}$ as a type 1-object, namely as a double sequence (See for instance \cite{simpson2}*{p.\ 13}), and the same holds for $\Phi(R)$.  

\begin{thm}\label{dingdong}
In $\RCAo$, we have the explicit equivalence $\UCOH\asa(\mu^{2})$. 
\end{thm}
\begin{proof}
For the reverse implication, since cohesiveness is an arithmetical property, it is easy to build the functional $\Phi$ from $\UCOH$ assuming $(\exists^{2})$.

\medskip

For the forward implication, consider $\UCOH$ and apply \textsf{QF-AC}$^{1,0}$ to the first conjunct of \eqref{dfghj} to obtain $\Xi^{2}$ such that $(\forall R^{0\di 1}, k^{0})[\Xi(R,k)>k \wedge \Xi(R,k)\in \Phi(R)(1)]$.  
Define $\UCOH^{+}$ as the resulting formula but starting with $(\exists^{\st} \Phi, \Xi)(\forall^{\st} R^{0\di 1})$ and the addition that $\Phi$ and $\Xi$ are standard extensional.  
Note that we can decide which disjunct holds (for given $i$) in the second conjunct of \eqref{dfghj} by checking if $\Xi(R, \Phi(R)(2)(i))\in R(i)$.  For standard $R, i$, the latter only involves standard objects.  
We now prove $\UCOH^{+}\di \paai$, from which the theorem follows by applying the algorithms $\mathcal{B}$ and $\mathcal{A}$ using Corollary \ref{consresultcor}.  

\medskip

Now assume $\UCOH^{+}$ and suppose there is standard $h^{1}$ such that $(\forall^{\st} n)h(n)=0 \wedge (\exists m)h(m)\ne 0$.
Suppose for some fixed standard $R$, there is standard $i_{0}$ such that the first disjunct holds in the second conjunct of \eqref{dfghj}.  
Now define $R'$ as follows: $k\in R'(j) \asa [k\in R(j)\wedge (\forall n\leq \max(j,k))h(n)=0 ]$.  
Clearly, $R'$ is standard and we have $R\approx_{0\di 1}R'$, implying $\Phi(R)\approx_{1\times 1}\Phi(R')$.    
In particular, $\Phi(R)(2)(i_{0})=_{0}\Phi(R')(2)(i_{0})$, and $\Phi(R)(1)\approx_{1} \Phi(R')(1)$.  However, then the first disjunct holds in the second conjunct of \eqref{dfghj} for $R', i_{0}$ too, since 
$\Xi(R', \Phi(R')(2)(i_{0}))\in R'(i_{0})$ is equivalent to $\Xi(R, \Phi(R)(2)(i_{0}))\in R(i_{0})$.  
However, now let $m_{0}$ be such that $h(m_{0})\ne 0$ and take $m_{0}<l_{0}\in \Phi(R')(1)$.  Clearly, $l_{0}>\Phi(R')(2)(i_{0})$ as the first number is infinite and the second finite.  But then $l_{0}\in R'(i_{0})$ by \UCOH$^{+}$, which is impossible by the definition of $R'$.         
A similar procedure leads to a contradiction in case the second disjunct holds in the second conjunct of \eqref{dfghj} for some standard $i_{0}$.  
In light of these contradictions, the implication $\UCOH^{+}\di \paai$ follows.  
\end{proof}
While Ramsey's theorem for pairs $\RT_{2}^{2}$ does not imply $\WKL$ (See e.g.\ \cites{liu, dslice}), the uniform versions are equivalent.  
\begin{cor}
In $\RCAo$, we have the explicit equivalence $\textup{\textsf{URT}}_{2}^{2}\asa \textup{\textsf{UWKL}}$.
\end{cor}
\begin{proof}
The implication $\textup{\RT}_{2}^{2}\di \COH$ is proved in \cite{dslice}*{6.32}.  This proof is clearly uniform (as also noted at the end of \cite{dslice}*{p.\ 85}), yielding \textsf{URT}$_{2}^{2}\di $ \UCOH, and the theorem implies the forward implication, since $(\exists^{2})\asa \UWKL$ (\cite{kohlenbach2}).  
By \cite{yamayamaharehare}*{Theorem 4.2}, the reverse implication follows.    
\end{proof}
Next, we study the cohesive version of \ADS.  Recall the definition of a stable order from Definition \ref{DISQUE}.  
Denote by \textsf{CADS} the statement that every infinite linear order has a stable suborder.  The connection between \textsf{CADS} and cohesiveness is discussed between \cite{dslice}*{9.17-9.18}.
Now let \textsf{UCADS} be the `fully' uniform version of \textsf{CADS} as follows. 
\bdefi[\textsf{UCADS}] There is $\Phi^{1\di (1\times 1)}$ such that for infinite linear orders $X^{1}$, $Y\equiv\Phi(X)(1)$ is a stable suborder of $X$ and $\Phi(X)(2)$ witnesses this, i.e.\ for $y^{0}\in Y$:
\be\label{dark}
  (\forall w^{0})(y\leq_{Y} w \di w\leq_{Y} \Phi(X)(2)(y)) \vee  (\forall v^{0})(y\geq_{Y} v \di v\geq_{Y} \Phi(X)(2)(y)). 
\ee
\edefi
\begin{thm}
In $\RCAo$, we have the explicit equivalence $\textup{\textsf{UCADS}}\asa (\mu^{2})$.
\end{thm}
\begin{proof}
The reverse implication is immediate in light of Theorem \ref{dingdong} and the uniformity of the proofs of \cite{schirfeld}*{Prop.\ 1.4 and 2.9}.  
For the forward implication, we proceed as in the proof of Theorem \ref{dingdong}: Consider \textsf{UCADS} and apply \textsf{QF-AC}$^{1,0}$ to the 
formula expressing that $\Phi(X)(1)$ is infinite to obtain $\Xi^{2}$ such that $(\forall X^{1},k^{0})[\Xi(X,k)>_{0}k \wedge \Xi(X,k)\in \Phi(X)(1)]$.  
Define $\textsf{UCADS}^{+}$ as the resulting formula but starting with $(\exists^{\st} \Phi, \Xi)(\forall^{\st} X^{1})$ and the addition that $\Phi$ and $\Xi$ are standard extensional.  

\medskip

Now assume $\textsf{UCADS}^{+}$ and suppose $h$ is a counterexample to $\paai$.  Consider again the orders $<_{0}$ and $\prec$ from the proof of Theorem \ref{JUDAS}.  
Since $<_{0}~\approx_{1}~\prec$ (again with some abuse of notation), we have $\Phi(<_{0})\approx_{1\times 1} \Phi(\prec)$.  
Now take standard $n_{0}\in \Phi(<_{0})(1)$ (which exist by the standardness of $\Xi$ and also satisfies $n_{0}\in \Phi(\prec)(1)$ by standard extensionality) 
and consider the standard number $\Phi(<_{0})(2)(n_{0})=_{0}\Phi(\prec)(2)(n_{0})$, the latter equality again by standard extensionality.  
However, by the infinitude of $\Phi(<_{0})(1)$ (resp.\ of $\Phi(\prec)(1)$) only the second (resp.\ first) disjunct of \eqref{dark} can hold for $<_{0}$ (resp.\ for $\prec$).    
Then, the second (resp.\ first) disjunct of \eqref{dark} for $<_{0}$ (resp.\ $\prec$) implies $n_{0}\geq_{0}\Phi(<_{0})(2)(n_{0})$ (resp.\ $n_{0}\preceq\Phi(\prec)(2)(n_{0})$).  
Since all objects are standard, we obtain $n_{0}=_{0}\Phi(<_{0})(2)(n_{0})=_{0}\Phi(\prec)(2)(n_{0})$.  However, then $\Phi(<_{0})(1)\approx_{1}\Phi(\prec)(1)$ is impossible 
as the `overlap' between the latter two orders is a singleton, namely $\{n_{0}\}$.  
\end{proof}
In \cite{schirfeld}*{Prop.\ 2.9}, a uniform proof of \textsf{CADS} from \textsf{CRT}$_{2}^{2}$, a cohesive version of \RT$_{2}^{2}$, is presented.  
Hence, it follows that the (fully) uniform version of \textsf{CRT}$_{2}^{2}$ is also equivalent to $(\exists^{2})$.  Finally, we discuss a connection between our results and \cite{bennosam}.    
\begin{rem}[Alternative approach]\rm
The above \emph{non-explicit} results can also be obtained in a different way:  It is established in \cite{bennosam}*{Cor.\ 12} that $(\exists^{2})\asa \paai$ over a suitable (nonstandard) base theory.  The essential ingredient in the latter system
is \emph{parameter-free Transfer} \textsf{PF-TP}$_{\forall}$, i.e.\ Nelson's axiom \emph{Transfer} (See Section~\ref{base}) where the formulas $\varphi$ have no parameters.  In contrast to \emph{Transfer}, \emph{parameter-free Transfer} does not carry any logical strength.  
However, the principle $\UDNR$ is parameter-free, implying that the functional $\Psi$ from the former is \emph{standard}, assuming $\textsf{PF-TP}_{\forall}$.  Similarly, the principle `There is $\Psi, \Xi$ such that $\UDNR(\Psi)$ and $\Xi$ is an extensionality functional for $\Psi$' does not have any parameters, and $\textsf{PF-TP}_{\forall}$ yields that $\Psi$ and $\Xi$ are standard.  However, the standardness of $\Xi$ \emph{also implies that $\Psi$ is standard extensional}.  Hence, Theorem \ref{UDNRPLUS} and \cite{bennosam}*{Cor.\ 12} immediately yield that $\UDNR\asa \paai\asa (\exists^{2})$, assuming $\textsf{PF-TP}_{\forall}$.  These equivalences were proved in a conservative extension of $\RCAo$, implying that the latter proves $(\exists^{2})\asa \UDNR$.      
\end{rem}
\subsection{Classifying the strong Tietze extension theorem}\label{ctietze}
In this section, we study a uniform version of the Tietze (extension) theorem.  Non-uniform versions of the Tietze theorem are studied in \cite{simpson2}*{II.7} and \cite{withgusto}.  
We are interested in the `strong' Tietze theorem \cite{withgusto}*{6.15.(5)} since it implies $\DNR$ and is implied by $\WKL$ (See \cite{withgusto}*{\S6}).  Furthermore, Montalb\'an lists the status of the Tietze theorem as an open question in Reverse Mathematics in \cite{montahue}*{Question 16}.  We will establish an explicit equivalence between $(\mu^{2})$ (and hence $\UWKL$ by \cite{kohlenbach2}*{\S3}) and the uniform strong Tietze theorem.  
We make essential use of Corollary \ref{essek}.         

\medskip

First of all, since the Tietze theorem from \cite{withgusto}*{6.15.(5)} is about uniformly continuous functions with a modulus, 
it does not really matter which definition of continuity is used by \cite{kohlenbach4}*{Prop.\ 4.4}.  
Thus, let $f^{1}\in C_{\textup{rm}}(X)$ mean that $f$ is continuous in the sense of Reverse Mathematics on $X$, i.e.\ as in \cite{simpson2}*{II.6.1} or \cite{withgusto}*{Def.~2.7}.  
Furthermore, let $\mathscr{C}(X)$ be the Banach space used in the Tietze theorem  \cite{withgusto}*{6.15.(5)} as defined in \cite{withgusto}*{p.\ 1454}.  
Finally, we use the same definition for closed and separably closed sets as in \cite{withgusto}.  
\begin{princ}[\textsf{UTIE}]
There is a functional $\Psi^{ (1\times 1)\di 1}$ such that for closed and separably closed sets $A\subseteq [0,1]$ and for $f\in C_{\textup{rm}}(A)$ with modulus of uniform continuity $g$, we have $\Psi(f,g,A)\in \mathscr{C}([0,1])$ and $f$ equals $\Psi(f,g,A)$ on $A$.   
\end{princ}
We also study the following uniform version of Weierstra\ss' (polynomial) approximation theorem.  The non-uniform version is equivalent to $\WKL$ by \cite{simpson2}*{IV.2.5}  
\begin{princ}[\textsf{UWA}]  There is $\Psi^{1\di 1}$ such that
\[\textstyle
(\forall f\in C_{\textup{rm}}[0,1])(\forall x^{1}\in [0,1], n^{0})\big[ \Psi(f)(n) \in \textup{POLY}\text{} \wedge |f(x)-\Psi(f)(n)(x)| <\frac{1}{2^{n}}  \big].
\]
\end{princ}
\begin{thm}
In $\RCAo$, we have the explicit equivalences $\textup{\textsf{UWA}}\asa \textup{\textsf{UTIE}}\asa (\mu^{2})$.  
\end{thm}
\begin{proof}
As in the proof of \cite{kohlenbach2}*{Prop.\ 3.14}, it is straightforward to obtain \textsf{UWA} using $(\exists^{2})$ from the associated non-uniform proof, even when $f$ is a type $1\di 1$ functional which 
happens to be $\eps$-$\delta$-continuous.  
Indeed, it is well-known that  $\lim_{n\di \infty}B_{n}(f)(x)=f(x)$ uniformly for $x\in [0,1]$, if $f$ is continuous on $[0,1]$ and $B_{n}(f)$ are the associated Bernstein polynomials (\cite{natan1}*{p.\ 6}).  
Using $(\exists^{2})$ it is then easy to define $\Psi(f)(n)$ as the least $N$ such $\sup_{x\in [0,1]} |B_{N}(f)(x)-f(x)|\leq \frac{1}{2^{2n+2}} $.   

\medskip

For the explicit implication \textsf{UTIE} $\di (\exists^{2})$, we will use of Corollary~\ref{essek} and \cite{withgusto}*{\S6}.  
In particular, we will `uniformise' the proof of \cite{withgusto}*{Lemma 6.17}.

\medskip

First of all, by \cite{withgusto}*{Lemma 6.17}, the strong Tietze theorem \cite{withgusto}*{6.15.(5)} implies \DNR.   
In this proof, a function $f$ defined on a set $C$ is constructed in $\RCA_{0}$ (See the proof of \cite{withgusto}*{Lemma 6.16}).  
This function satisfies all conditions of the strong Tietze theorem; In particular, it has a modulus of uniform continuity of $f$.
Applying \cite{withgusto}*{6.15.(5)}, one obtains $F\in \mathscr{C}[0,1]$,  
an extension of $f$ to $[0,1]$.  

\medskip

Secondly, by the definition of $\mathscr{C}(X)$ from \cite{withgusto}*{p. 1454},  $F$ is coded by a sequence of polynomials $p_{n}$ such that $\|p_{n}-F\|<\frac{1}{2^{2n+2}}$, and we can define $h(n):=\sharp(p_{n})$.  
The latter is then such that $(\forall e^{0})(h(e)\ne \Phi_{e}(e))$.  The case of $\DNR$ where $A\ne \emptyset$ is then straightforward.  Indeed, the initial function $f$ (from the proof of \cite{withgusto}*{Lemma~6.16}) is defined using a recursive counterexample to the Heine-Borel lemma.  
Such a counterexample can be found in \cite{simpson2}*{I.8.6} and clearly relativizes (uniformly) to any set $A$.  Let us use $f_{A}$ to denote the function $f$ obtained from the previous construction relative to the set $A$, and
let $C_{A}$ and $g_{A}$ be the relativized domain and modulus.      
Now let $\Psi$ be the functional from $\textsf{UTEI}$ and define 
$\Xi^{1\di 1}$ by 
\[
\Xi(A):=\sharp\big(\Psi(f_{A},g_{A}, C_{A})\big), 
\]
where $f_{A}$, $g_{A}$, and $C_{A}$ are as in the previous paragraph of this proof.  
In the same way as in the proof of \cite{withgusto}*{Lemma 6.17}, one proves that for any $A^{1}$, we have $(\forall e^{0})(\Xi(A)(e)\ne \Phi_{e}^{A}(e))$.  
However, this yields the explicit implication $\textsf{UTIE}\di \UDNR$ and the latter explicitly implies $(\mu^{2})$ by Corollary \ref{essek}.  

\medskip

Next, to prove the explicit implication \textsf{UWA} $\di$ \textsf{UTIE}, note that Simpson proves an effective version of the Tietze theorem in \cite{simpson2}*{II.7.5}.  
Following the proof of the latter, it is clear that there is a functional $\Phi$ in $\RCAo$ such that for closed and separably closed $A$ and $f\in C_{\textup{rm}}(A)$, the image $\Phi(f, g,A)\in {C}_{\textup{rm}}[0,1]$ is the extension of $f$ to $[0,1]$ provided by \cite{simpson2}*{II.7.5}.  For $\Psi$ as in \textsf{UWA}, the functional $\Psi(\Phi(f,g,A))$ is as required by \textsf{UTIE}.  
\end{proof}
Let \textsf{UTIE}$'$ and \textsf{UWA}$'$ be the versions of \textsf{UTIE} and \textsf{UWA} with the usual epsilon-delta definition of continuity instead of the Reverse Mathematics definition of continuity.  
The following corollary is immediate from the proof of the theorem.  
\begin{cor}
In $\RCAo$, we have $\textup{UWA}'\asa \textup{UTIE}'\asa (\exists^{2})$.
\end{cor}
In the proof of the theorem, we established \textsf{UTIE} $\di (\exists^{2})$ by showing that \textsf{UTIE} $\di$ \textsf{UDNR}, and then applying Corollary~\ref{essek}.  
The latter implication goes through because of the uniformity of the proof of $\DNR$ from the strong Tietze theorem (See \cite{withgusto}*{Lemma~6.17}).

\begin{rem}[The role of extensionality]\label{equ3}\rm
At the risk of stating the obvious, the axiom of extensionality is central in proving all above equivalences;  In particular, half of the explicit implications obtained above all have an extensionality functional `buit-in'.  
Hence, an approach to uniform computability \emph{not involving the axiom of extensionality} will yield different results.  It is a matter of opinion whether in the latter such `non-extensional framework', the glass is half-full (finer distinctions) or half-empty (more complicated picture).  In our opinion, it is remarkable how uniform our uniform classification has turned out.    
\end{rem}
\section{Taming the future Reverse Mathematics zoo}\label{GT}
In this secton, we formulate a general template for obtaining (explicit) equivalences between $(\mu^{2})$ and uniform versions of principles from the RM zoo.  
\subsection{General template}\label{tempie}  Our template is defined as follows.
\begin{tempo*}\rm
Let $T\equiv (\forall X^{1})(\exists Y^{1})\varphi(X,Y)$ be a  RM zoo principle and let $UT$ be $(\exists \Phi^{1\di 1})(\forall X^{1})\varphi(X,\Phi(X))$.  
To prove the \emph{explicit} implication $ UT\di (\mu^{2})$, execute the following steps:   
\begin{enumerate}[(i)]
\item Let $UT^{+}$ be $(\exists^{\st} \Phi^{1\di 1})(\forall^{\st} X^{1})\varphi(X,\Phi(X))$ where the functional $\Phi$ is additionally standard extensional.  We work in $\RCA_{0}^{\Lambda}+UT^{+}$.
\item Suppose the standard function $h^{1}$ is such that $(\forall^{\st}n)h(n)=0$ and $(\exists m)h(m)\ne0$, i.e.\ $h$ is a counterexample to $\paai$.  
\item For standard $V^{1}$, use $h$ to define standard $W^{1}\approx_{1} V$ such that $\Phi(W)\not\approx_{1}\Phi(V)$, i.e.\ $W$ is $V$ with the nonstandard elements changed sufficiently to yield a different image under $\Phi$.  \label{itemf}
\item The previous contradiction implies that $\RCA_{0}^{\Lambda}$ proves $UT^{+}\di \paai$.
\item Bring the implication from the previous step into the normal form\\ $(\forall^{\st}x)(\exists^{\st}y)\psi(x,y)$ ($\psi$ internal) using the algorithm $\mathcal{B}$ from Remark \ref{algob}.  \label{frink}
\item Apply the term extraction algorithm $\mathcal{A}$ using Corollary \ref{consresultcor}.  The resulting term yields the explicit implication $UT\di (\mu^{2})$.  \label{frink2}
\end{enumerate}  
The explicit implication $(\mu^{2})\di UT$ is usually straightforward;  Alternatively, establish $\paai\di UT^{+}$ in $\RCA_{0}^{\Lambda}$ and apply steps \eqref{frink} and \eqref{frink2}.  
\end{tempo*}
The algorithm $\RS$ is defined as the steps \eqref{frink} and \eqref{frink2} in the template, i.e.\ the application of the algorithms $\mathcal{B}$ and $\mathcal{A}$ to suitable implications.  
In Section \ref{robu}, we speculate why uniform principles $UT$ originating from RM zoo-principles are equivalent to $(\exists^{2})$ \emph{en masse}.  We conjecture a connection to Montalb\'an's notion of \emph{robustness} from \cite{montahue}.    

\medskip

Finally, the above template treats zoo-principles in a kind of `$\Pi_{2}^{1}$-normal form', for the simple reason that most zoo-principles are formulated in such a way.  
Nonetheless, it is a natural question, discussed in Section \ref{DX}, whether principles \emph{not} formulated in this normal form gives rise to uniform principles not equivalent to $(\exists^{2})$.    
Surprisingly, the answer to this question turns out to be negative.  

\subsection{Robustness and structure}\label{robu}\rm
In this section, we try to explain why our template works so well for RM zoo principles.  We conjecture a connection to Montalb\'an's notion of \emph{robustness} from \cite{montahue}.      

\medskip

First of all, standard computable functions are determined by their behaviour on the standard numbers (by the \emph{Use principle} from \cite{zweer}*{p.\ 50}), while e.g.\ a standard Turing machine may 
well halt at some infinite number (given e.g.\ the fan functional from \cite{kohlenbach2} or $\neg\paai$), i.e.\ non-computable problems, like the Halting problem for standard Turing machines, are not necessarily determined by the standard numbers.    

\medskip

Now in step \eqref{itemf}, the assumption $\neg\paai$ allows us to change the nonstandard part of a standard set $V^{1}$, resulting in standard $W^{1}\approx V$.  Since $\Phi(V)$ (resp.\ $\Phi(W)$) is not computable from $V$ (resp.\ $W$), the former depends on the nonstandard numbers in $V$ (resp.\ $W$).  However, making the nonstandard parts of $V$ and $W$ different enough, 
we can guarantee $\Phi(W)\not\approx_{1} \Phi(V)$, and obtain a contradiction with standard extensionality.  Hence, $\paai$ follows and so does $UT\di (\exists^{2})$.         
Alternatively, as noted in Remark \ref{dikko}, we can define standard $V$ and \emph{nonstandard} $W$ such that $V\approx_{1}W\wedge \Phi(V)\not\approx_{1}\Phi(W)$ without assuming $\neg\paai$.  
Hence $\Phi$ is not \emph{nonstandard} continuous and Kohlenbach has pointed out that a discontinuous function can be used to define $(\exists^{2})$ using \emph{Grilliot's trick} (See \cite{kohlenbach2}*{Prop.\ 3.7}).

\medskip

Secondly, note that step \eqref{itemf} crucially depends on the fact that we can modify the nonstandard numbers in the set $V$ \emph{without changing the standard numbers}, i.e.\ while guaranteeing $V\approx_{1} W$.  Such a modification is only possible for structures which are \emph{not} closed downwards:  For instance, our template will fail for the \emph{fan theorem} (See Section \ref{DX}), as the latter deals with (finite) binary trees, which are closed downwards.     
Of course, many of the zoo-principles have a distinct combinatorial flavour, which implies that the objects at hand exhibit little structure.  
Furthermore, as noted in Remark \ref{dikko}, this absence of structure directly translates into a (nonstandard) discontinuity in the input parameter $X$ in $(\forall X^{1})(\exists Y^{1})\varphi(X,Y)$.  

\medskip

Thirdly, in light of this absence of structure in principles of the RM zoo, we conjecture that \emph{robust} theorems (in the sense of \cite{montahue}*{p.\ 432}) are (exactly) 
those which deal with \emph{mathematical objects with lots of structure} like trees, continuous functions, metric spaces, et cetera.  These theorems are also (exactly) those which are continuous in their input parameter(s).      
In particular, the presence of this structure `almost guarantees' a place in one of the Big Five categories.  
The non-robust theorems, by contrast, deal with objects which exhibit little structure (and hence can be discontinuous in their input parameters), and for this reason \emph{have the potential} to fall outside the Big Five and in the RM zoo.  
However, as we observed in the previous paragraph, the absence of structure in RM zoo principles, is exactly what makes our template from Section \ref{tempie} work.  

\medskip

In conclusion, what makes the \emph{principles in the \textup{RM} zoo exceptional} (namely the presence of little structure on the objects at hand) guarantees that the \emph{uniform versions of the \textup{RM} zoo principles are non-exceptional} (due to the fact that the above template works form them).  

\section{Converse Mathematics}\label{DX}
In this section, we classify the uniform versions of the \emph{contrapositions} of zoo-principles.
This study is motivated by the question whether the template from Section~\ref{tempie} `always' works, i.e.\ perhaps we can find 
counterexamples to this template by studying \emph{contrapositions} of zoo-principles, as these do not necessarily have a $\Pi_{2}^{1}$-structure?  We first discuss this motivation in detail.    

\medskip

First of all, the weak K\"onig's lemma (\WKL) is rejected in all varieties of constructive mathematics, while the (classical logic) contraposition of \WKL, called the \emph{fan theorem} is accepted in Brouwer's intuitionistic mathematics (See e.g.\ \cite{brich}*{\S5}).  
This difference in constructive content is also visible at the uniform level:  The uniform version of WKL satisfies the template from the previous section, and is indeed equivalent to arithmetical comprehension, while the uniform version of the fan theorem is not stronger 
than $\WKL$ itself. (See \cites{firstHORM, kohlenbach2}).
Hence, we observe that, from the constructive and uniform point of view, a principle can behave rather differently compared to its contraposition.    

\medskip

Secondly, the template from Section \ref{tempie} would seem to work for any $\Pi_{2}^{1}$-zoo principle $T\equiv (\forall X^{1})(\exists Y^{1})\varphi(X,Y)$ and the associated `obvious' uniform version $UT\equiv (\exists \Phi^{1\di 1})(\forall X^{1})\varphi(X,\Phi(X))$.  
Nonetheless, while $UT$ is the \emph{most natural} uniform version of $T$ (in our opinion), there sometimes exists an \emph{alternative} uniform version of $T$, similar to the uniform version of the fan theorem.  
With regard to examples, the principle $\ADS$ from Section \ref{padis} is perhaps the most obvious candidate, while various Ramsey theorems can also be recognised as suitable candidates.  

\medskip

In conclusion, it seems worthwhile investigating the uniform versions of contra-posed zoo-principles, inspired by the difference in behaviour of the fan theorem and weak K\"onig's lemma.  
However, somewhat surprisingly, we shall only obtain principles equivalent to arithmetical comprehension, i.e.\ our study will not yield exceptions to our observation that the RM zoo disappears at the uniform level.  
  
\subsection{The contraposition of $\ADS$}\label{CONTRADS}
In this section, we study the uniform version of the \emph{contraposition} of \ADS.  Recall that $\ADS$ states that every infinite linear order either has an ascending or a descending chain.  
Hence, the contraposition of $\ADS$ is the statement that if a linear order has no ascending and descending sequences, then it must be finite, as follows:
\begin{align}
(\forall X^{1})\big[\textup{LO}(X)\wedge & (\forall x_{(\cdot)}^{1}\in \textup{Seq}(X))(\exists n^{0}, k^{0})(x_{n}\leq_{X} x_{n+1} \wedge x_{k}\geq_{X} x_{k+1})\notag\\
&\di (\exists l^{0}, k^{0}\in \field{(X)})(\forall m^{0}\in \field(X))(k\leq_{X}m\leq_{X} l)\big].\label{kund}
\end{align}
%
%
By removing \emph{all} existential quantifiers, we obtain the following alternative uniform version of \ADS.
\begin{princ}[$\UADS_{2}$]
There is $\Phi^{3}$ such that for all linear orders $X^{1}$ and $g^{2}$
\begin{align}
 (\forall x_{(\cdot)}^{1}\in \textup{Seq}(X))&(\exists n^{0}, k^{0}\leq g(x_{(\cdot)}))(x_{n}\leq_{X} x_{n+1} \wedge x_{k}\geq_{X} x_{k+1})\notag\\
&\di (\forall m^{0}\in \field(X))(\Phi(X,g))(1)\leq_{X}m\leq_{X} \Phi(X,g)(2)).  \label{gohan}
\end{align}
\end{princ}
\begin{thm}\label{ravvy}
In $\RCAo$, we have the explicit equivalence $\UADS_{2}\asa (\mu^{2})$.  
\end{thm}
\begin{proof}
The reverse direction is immediate since $(\mu^{2})$ implies $\ADS$ and the upper and lower bounds to $\leq_{X}$ in the consequent of \eqref{kund} can be found using the same search operator.   
 For the forward direction, we shall apply the template from Section \ref{tempie}.  
Hence, let \UADS$_{2}^{+}$ be as in the template and fix standard $X_{0}\ne \emptyset$ and $g_{0}$ such that the antecedent of \UADS$_{2}^{+}$ holds.   
Then $\Phi(X_{0}, g_{0})$ is standard and consider the standard function $h^{2}_{0}$ which is constant and always outputs $\Phi(X_{0}, g_{0})(1)+\Phi(X_{0}, g_{0})(2)+4$.  
Clearly, we have:
\be\label{corkk}
 (\forall x_{(\cdot)}^{1}\in \textup{Seq}({X_{0}}))(\exists n^{0}, k^{0}\leq h_{0}(x_{(\cdot)}))(x_{n}\leq_{X_{0}} x_{n+1} \wedge x_{k}\geq_{X_{0}} x_{k+1}),
\ee
as there are less than $\Phi(X_{0}, g_{0})(1)+\Phi(X_{0}, g_{0})(2)+2$ distinct elements in the finite linear order induced by $X_{0}$, by the consequent of \UADS$_{2}^{+}$.  
Indeed, by the definition of linear order (\cite{simpson2}*{V.1.1}), if $x \leq_{X} y\wedge x\geq_{X} y $, then $x=_{0} y$, i.e.\ equality in the sense of $X$ is equality on the natural numbers.  
By \eqref{corkk}, the associated consequent of \UADS$_{2}^{+}$ also follows for $\Phi(X_{0}, h_{0})$.     

\medskip

Now suppose that $\paai$ is false, i.e.\ there is standard function $h_{1}^{1}$ such that $(\forall^{\st}n)(h_{1}(n)=0)$ and $ (\exists n_{0})h_{1}(n)\ne0$.  
Following \cite{simpson2}*{V.1.1}, define the standard set $Y_{0}$ by adding to $X_{0}$ the pairs $(x, m_{0})$ for $x\in \textup{field}(X_{0})$ and where $m_{0}$ is such that $(\forall i<m_{0})h_{1}(i)=0\wedge h_{1}(m_{0})\ne 0$.  
Intuitively speaking, the standard set $Y_{0}$ represents the linear order $X_{0}$ with a `point at infinity' $m_{0}$ added (in a standard way, thanks to $h_{1}$).  
Since the order induced by $Y_{0}$ is only a one-element extension of the order induced by $X_{0}$, we also have 
\[
 (\forall x_{(\cdot)}^{1}\in \textup{Seq}({Y_{0}}))(\exists n^{0}, k^{0}\leq h_{0}(x_{(\cdot)}))(x_{n}\leq_{Y_{0}} x_{n+1} \wedge x_{k}\geq_{Y_{0}} x_{k+1}),
\]  
i.e.\ the antecedent of \UADS$_{2}^{+}$ holds for $Y_{0}$ and $h_{0}$.  Hence, the order induced by $Y_{0}$ is bounded by $\Phi(Y_{0}, h_{0})$ as in the consequent of \UADS$_{2}^{+}$.  
However, by definition, we have $X_{0}\approx_{1}Y_{0}$, implying $\Phi(X_{0}, h_{0})=_{0^{*}}\Phi(Y_{0}, h_{0})$.  By the latter, we cannot have $m_{0} \leq_{Y_{0}}\Phi(Y_{0}, h_{0})$ for the unique (and necessarily infinite) element 
$m_{0}\in\field (Y_{0})\setminus\field(X_{0})$, i.e.\ a contradiction.        
Hence, we obtain $\UADS_{2}^{+}\di \paai$ and applying $\RS$ finishes the proof.  
\end{proof}
In the previous proof, we added the `point at infinity' $m_{0}$ to the finite linear order induced by $X_{0}$;  Such a modification is only possible for structures which are \emph{not} closed downwards.  
In particular, the above approach clearly does not work for theorems concerned with trees, like e.g.\ the fan theorem.  On the other hand, we can easily obtain a version of the previous theorem for e.g.\ the chain-antichain principle \CAC, and of course for stable versions of the latter and of \ADS.

\subsection{The contraposition of Ramsey theorems}
In this section, we study the well-known \emph{Ramsey's theorem for pairs} $\RT_{2}^{2}$.  The latter is the statement that every colouring with two colours of all two-element sets of natural numbers must have an infinite homogenous subset, i.e.\ of the same colour.  
Now, $\RT_{2}^{2}$ has an equivalent version (See \cite{dslice}*{\S6})) of which the contraposition has the `right' syntactic structure, namely similar to the fan theorem.  
Thus, consider the following principle.  
\begin{princ}[Contraposition of $\RT_{2}^{2}$]\label{donkey}
\begin{align}
(\forall X^{1},c^{1}:[X]^{2}\di 2)\Big[ (\forall H^{1}\subseteq X)(\forall i<2)&\big[(\forall s^{0}\in [H]^{2})(c(s)=i)\label{posmoking} \\
&  \di   H \textup{ is finite}\big] \di   X \textup{ is finite}\Big].\notag
\end{align}
\end{princ}
Here, `$Z^{1}$ is finite' is short for $(\exists n^{0})(\forall \sigma^{0^{*}})\big[(\forall i<|\sigma|)(\sigma(i)\in Z^{1})\di |\sigma|\leq n \big]$.  We also abbreviate the previous formula by $(\exists n^{0})(|Z^{1}|\leq n)$, where obviously $|Z^{1}|\leq n$ is a $\Pi_{1}^{0}$-formula.  Note that we used the usual notation $[H]^{n}$ for the set of $n$-element subsets of $H$, which of course has nothing to do with the typing of variables.   

\medskip
  
Based on the previous principle, define \textsf{URTP}$_{2}$ as the following principle.
\begin{princ}
There is $\Phi^{3}$ such that for all $g^{2}, X^{1},c^{1}:[X]^{2}\di 2$, we have 
\be\label{frug}
 (\forall H \subseteq X)(\forall i<2)\big[(\forall s\in [H]^{2})c(s)=i   \di   |H|\leq g(H) \big] \di   |X|\leq \Phi(X, g, c).
\ee
\end{princ}
Note that $g$ does not depend on $i$, as the quantifier $(\forall i<2)$ can be brought inside the square brackets to obtain $(\forall s^{0}\in [H]^{2})(c(s)=0)\vee (\forall t^{0}\in [H]^{2})(c(t)=1)$.     
\begin{thm}\label{dirko}
In $\RCAo$, we have the explicit equivalence $(\mu^{2})\asa \textup{\textsf{URTP}}_{2}$.
\end{thm}
\begin{proof}
The forward direction is immediate as $(\mu^{2})$ implies $\RT_{2}^{2}$ and the upper bound to $|X|$ in the former's contraposition can be found using this search operator.    
For the reverse direction, we work following the template from Section \ref{tempie}.  
Hence, consider $\textup{\textsf{URTP}}_{2}^{+}$ and let $g^{2}_{0}, X^{1}_{0},c^{1}_{0}:[X_{0}]^{2}\di 2$ be standard objects 
such that the antecedent of \eqref{frug} holds and hence $|X_{0}|\leq\Phi(X_{0}, g_{0}, c_{0})$, where $X_{0}\ne \emptyset$.  Now define $h_{0}^{2}$ to be the functional which is constantly $\Phi(X_{0}, g_{0}, c_{0})+1$, and note that:  
\[
 (\forall H \subseteq X_{0})(\forall i<2)\big[(\forall s\in [H]^{2})(c_{0}(s)=i)   \di   |H|\leq h_{0}(H) \big],
\]
as $H\subseteq X_{0}$ implies that $|H|\leq |X_{0}|$.  By \textsf{URTP}$_{2}$, we also have $|X_{0}|\leq \Phi(X_{0}, h_{0}, c_{0})$.  

\medskip

Now suppose $\paai$ is false, i.e.\ there is some standard $h_{1}$ such that $(\forall^{\st}n)h_{1}(n)=0$ and $ (\exists m_{0})h_{1}(m_{0})$, and define the standard set $Y_{0}$ as $X_{0}\cup \{m_{0}, m_{0}+1, \dots, m_{0}+\Phi(X_{0}, h_{0}, c_{0}) \}$, where $m_{0}$ is the least number $k$ such that $h_{1}(k)\ne 0$.    
Now define the standard colouring $d_{0}^{1}$ as follows:  $d_{0}(s)$ is $0$ if both elements of $s$ are at least $m_{0}$, $1$ if one element of $s$ is at least $m_{0}$ and the other one is not, and $c_{0}(s)$ otherwise.  
By the definition of $Y_{0}$ and $d_{0}$, we have
\be\label{flacker}
 (\forall H \subseteq Y_{0})(\forall i<2)\big[(\forall s\in [H]^{2})(d_{0}(s)=i)   \di   |H|\leq h_{0}(H) \big],
\ee
as for $H\subseteq Y_{0}$ with more than $\Phi(X_{0}, g_{0}, c_{0})+1$ elements, the set $H$ is not homogenous for $d_{0}$.
By \textsf{URTP}$_{2}$, we obtain $|Y_{0}|\leq \Phi(Y_{0}, h_{0}, d_{0})$, but we also have $\Phi(Y_{0}, h_{0}, d_{0})=\Phi(X_{0}, h_{0}, c_{0})$ by standard extensionality since $X_{0}\approx_{1} Y_{0}$ and $c_{0}\approx_{1}d_{0}$.  
However, $Y_{0}$ by definition has more elements than $\Phi(X_{0}, h_{0}, c_{0})$, a contradiction.  Hence, we have $\textup{\textsf{URTP}}_{2}^{+}\di\paai$ and applying $\RS$ finishes the proof.  
\end{proof}

\subsection{Contraposition of thin and free set theorems}
In this section, we again study the thin- and free set theorems from \cite{freesets}.  
These results are similar to those in the previous two sections, hence our treatment will be brief.  
Notations are as in \cite{freesets}, except that we write $f:[X]^{k}\di N$ instead of $f:[X]^{k}\di \N$.  

\medskip

Recall the equivalent version of Ramsey's theorem from \cite{dslice}*{\S6} in Principle~\ref{donkey}.  
Because of the extra set parameter $X^{1}$ in the latter, \eqref{posmoking} is amenable to our treatment as in Theorem \ref{dirko}.  
As it turns out, the free and this set theorems also have 
such equivalent versions by \cite{freesets}*{Lemma 2.4 and Corollary 3.6}.   

\medskip

For instance, by the aforementiond lemma, $\textsf{FS}(k)$, the free set theorem for index $k$, is equivalent to the statement that for every infinite set $X^{1}$ and $f^{1}:[X]^{k}\di N$, there is infinite $A^{1}\subset X$ which is free for $f$. 
The contraposition of the latter is:
\begin{align}
(\forall X^{1},f^{1}:[X]^{k}\di N)\Big[ (\forall A^{1}\subseteq X)\big[(\forall s^{0}\in [A]^{k})&(f(s)\not \in A \vee f(s)\in s) \label{nosmoking}\\
&  \di   H \textup{ is finite}\big] \di   X \textup{ is finite}\Big].\notag
\end{align}
which is neigh identical to Principle \ref{donkey} for $k=2$.  Now let \textsf{UFSP}$_{k}$ be the uniform version of \eqref{nosmoking} similar to \textsf{URTP}$_{2}$.     
Similar to Theorem \ref{dirko}, one proves the following.
\begin{thm}
In $\RCAo$, we have the explicit equivalence $(\mu^{2})\asa \textup{\textsf{UFSP}}_{2}$.
\end{thm}
The version of the thin set theorem from \cite{freesets}*{Corollary 3.6} is not so elegant, hence we do not consider it.  
We finish this section with some concluding remarks
\begin{rem}\label{finalk}\rm
First of all, Kohlenbach claims in \cite{kohlenbach2}*{\S1} that $(\exists^{2})$ sports a rich and very robust class of equivalent principles, 
which seems to be `more than' confirmed by the above results, especially those in this section.  

\medskip

Secondly, if one were to categorise principles according to robustness \emph{at the uniform level}, $\ADS$ and other principles studied in this section would rank very high, as even their contrapositions give rise to uniform principles equivalent to $(\exists^{2})$.  
By contrast, $\WKL$ would rank lower, as the uniform version of the fan theorem, the classical contraposition of $\WKL$, is not stronger than $\WKL$, as discussed in the first part of this section.  
In other words, $\ADS$ is exceptional in Friedman-Simpson-style RM, while it is not in the aforementioned `uniform' categorisation.     
\end{rem}
\subsection{Motivation for higher-order Reverse Mathematics}\label{whereohwhere}
The reader unaccustomed to higher-order arithmetic may deem higher-order principles like $\UDNR$ unnatural, compared to e.g.\ second-order arithmetic.  
We now argue that, at least from the point of view of second-order RM, higher-order RM is also natural.  It should also be mentioned that Montalb\'an includes higher-order RM among the `new avenues for RM' in \cite{montahue}.

\medskip

First of all, Fujiwara and Kohlenbach have established the connection (and even equivalence in some cases) between (classical) uniform existence as in $UT$ and intuitionistic provability (\cites{fuji1,fuji2}).  
Hence, the investigation of uniform principles like $\UDNR$ may be viewed as the (second-order) study of intuitionistic provability.    

\medskip

Secondly, the author shows in in \cite{samimplicit} that higher-order statements are implicit in 
(second-order) RM-theorems concerning continuity, due to the special nature of the RM-definition of continuity.  
In particular, consider the statement 
\begin{quote}
All continuous functions on Canter space are uniformly continuous.
\end{quote}
Let (H) be the previous statement \emph{with continuity as in the \textup{RM}-definition}.  One can\footnote{The proof takes place in $\RCAo+\QFAC^{2,0}$, a conservative extension of $\RCA_{0}$ by \cite{hunterphd}*{\S2.1.2}.} then prove (H)$\asa$(UH), where:
\begin{quote}
There is a functional which witnesses the uniform RM-continuity on Cantor space of any RM-continuous function.  \hfill (UH)
\end{quote}
From the treatment in \cite{samimplicit}, it is clear that the functional in (UH) can only be obtained because the RM-definition of continuity greatly reduces quantifier complexity.  
In conclusion, higher-order RM is already implicit in second-order RM \emph{due to the RM-definition of continuity involving codes}.    
Similar results are in \cite{firstHORM,sambrouwt}.  

\medskip

Thirdly, RM can be viewed as a classification based on \emph{computability}: Theorems provable in $\RCA_{0}$ are part of `computable mathematics';  An equivalence between a theorem and a Big Five system classifies 
the computational strength of the theorem, as the Big Five have natural formulations in terms of computability.  Furthermore, as noted by Simpson in \cite{simpson2}*{I.8.9 and IV.2.8}, theorems are analysed in RM `as they stand', in contrast to constructive mathematics, where extra conditions are added to enforce a constructive solution.  In other words, the goal of RM is not to enforce computability onto theorems, but to classify how `non-computable' the latter are.             

\medskip

In light of the previous, it is a natural question whether there are \emph{other natural ways} of classifying 
theorems of ordinary mathematics.  As noted in \cite{firstHORM, sambrouwt}, the study of uniform versions of theorems constitutes a classification based on the central tenet of Feferman's \emph{Explicit Mathematics} (See \cite{feferman2,fefmar,fefmons}), which is:
\begin{center}  
\emph{A proof of existence of an object yields a procedure to compute said object}.
\end{center}    
Indeed, in the same way as the RM-classification is based on the question which axioms (and hence `how much' non-computability) are necessary to prove a theorem, the study of uniform versions of theorems is motivated by the following question:  
\begin{center}
\emph{For a given theorem $T$, what extra axioms are needed to compute the objects claimed to exist by $T$?}
\end{center}
Similar to RM, we do not enforce the central tenet of Explicit Mathematics in higher-order RM: We measure `how much extra' is needed to obtain $UT$, the uniform version of $T$ where a functional witnesses the existential quantifiers.  
\section{Conclusion}\label{conc}
In conclusion, by establishing the template and associated algorithm $\RS$ in Section \ref{tempie}, we have exhibited a hitherto unknown `computational aspect' of Nonstandard Analysis.  
In particular, we have shown that for a theorem $T$ from the RM zoo, to obtain the explicit equivalence $UT\asa (\mu^{2})$ for the associated uniform version $UT$, we can just apply $\RS$ to the proof of the nonstandard equivalence $UT^{+}\asa \paai$.  
This conclusion suggests the following observations.
\begin{enumerate}
\item The Reverse Mathematics of Nonstandard Analysis gives rise to \emph{explicit} equivalences in classical Reverse mathematics \emph{without the need to actually construct the terms in the explicit equivalence}.      
\item Nonstandard Analysis carries plenty of computational content, in direct contrast to the claims made by e.g.\ Bishop (See \cite{kluut}*{p.\ 513} and \cite{kuddd}, which is the review of \cite{keisler3}) and Connes (See \cite{conman2}*{p.\ 6207} and \cite{conman}*{p.\ 26}) .
\item To extract more computational information from Nonstandard Analysis, we should study which notions (like continuity, Riemann integration, compactness, et cetera) can be brought into the normal form from Corollary~\ref{consresultcor}.  
As will be shown in \cite{samdaman}, this turns out to be a very large class.
\end{enumerate}

\begin{ack}\rm
This research was supported by the following funding bodies: FWO Flanders, the John Templeton Foundation, the Alexander von Humboldt Foundation, and the Japan Society for the Promotion of Science.  
The author expresses his gratitude towards these institutions. 
The author would like to thank Ulrich Kohlenbach, Benno van den Berg, Steffen Lempp, Paulo Oliva, Paul Shafer, Mariya Soskova, Vasco Brattka, and Denis Hirschfeldt for their valuable advice.  
\end{ack}

\bye